\documentclass[10pt, oneside]{amsart}

\title{Global transfer systems and endopurity of subgroups}

\author[Balchin]{Scott Balchin}
 \address{Mathematical Sciences Research Centre, Queen's University Belfast, UK}
  \email{s.balchin@qub.ac.uk}

\author[Barrero]{Miguel Barrero}
 \address{Max Planck Institute for Mathematics, Vivatsgasse 7, 53111 Bonn,
Germany}
 
  \email{barrero@mpim-bonn.mpg.de}

\usepackage[parfill]{parskip}

\usepackage[english]{babel}
\usepackage{mathtools} 
\usepackage[utf8]{inputenc}
\usepackage[T1]{fontenc}
\usepackage{amsmath, amsthm, amsfonts, mathrsfs, amssymb}
\usepackage{tikz-cd}
\usepackage{quiver}
\usepackage{caption}
\usepackage{pgfplots}
\pgfplotsset{compat=1.18}
\usepackage[shortlabels]{enumitem}
\usepackage[mathscr]{eucal}
\usepackage[all]{xy}
\xyoption{line}

\usepackage[a4paper, textheight = 24.5cm, textwidth = 16cm, top = 2.5cm, centering]{geometry}
\usepackage[final]{microtype}

\usepackage[backend=biber, style=alphabetic, maxnames=9,minnames=3, maxalphanames=3, doi=false, isbn=false, url=false, eprint=false]{biblatex}
\makeatletter
\newcommand{\thickhline}{%
    \noalign {\ifnum 0=`}\fi \hrule height 2pt
    \futurelet \reserved@a \@xhline
}

\usepackage{hyperref}

\hypersetup{
    colorlinks,
    linkcolor=red!80!black,
    citecolor = red!80!black
}

\usepackage{zref-clever}
\zcsetup{cap=true}
\AddToHook{env/thm/begin}{%
\zcsetup{countertype={thm=theorem}}}
\AddToHook{env/coro/begin}{%
\zcsetup{countertype={thm=corollary}}}
\AddToHook{env/lemm/begin}{%
\zcsetup{countertype={thm=lemma}}}
\AddToHook{env/prop/begin}{%
\zcsetup{countertype={thm=proposition}}}
\AddToHook{env/defi/begin}{%
\zcsetup{countertype={thm=definition}}}
\AddToHook{env/rem/begin}{%
\zcsetup{countertype={thm=remark}}}
\AddToHook{env/ex/begin}{%
\zcsetup{countertype={thm=example}}}

\newtheorem{thm}{Theorem}[section]
\newtheorem{coro}[thm]{Corollary}
\newtheorem{lemm}[thm]{Lemma}
\newtheorem{prop}[thm]{Proposition}

\newtheorem*{thmintro}{Theorem}
\theoremstyle{definition}
\newtheorem{defi}[thm]{Definition}

\newtheorem*{defintro}{Definition}
\newtheorem*{conjintro}{Conjecture}
\theoremstyle{remark}
\theoremstyle{remark}

\theoremstyle{remark}
\newtheorem{rem}[thm]{Remark}
\theoremstyle{remark}
\newtheorem{ex}[thm]{Example}

\newtheorem{constr}[thm]{Construction}
\zcRefTypeSetup{constr}{
Name-sg = Construction ,
name-sg = construction ,
Name-pl = Constructions ,
name-pl = constructions ,
}
\AddToHook{env/constr/begin}{%
\zcsetup{countertype={thm=constr}}}

\usepackage{listings}

\newcommand{\conlat}{\underline{\mathfrak{B}}}
\renewcommand{\leq}{\leqslant}
\renewcommand{\geq}{\geqslant}

\DeclareMathOperator{\Sub}{Sub}

\newcommand{\fami}{\mathcal{F}}

\newcommand{\la}{\langle}
\newcommand{\ra}{\rangle}

\DeclareMathOperator{\SL}{SL}

\DeclareMathOperator{\Img}{Im}

\DeclareMathOperator{\id}{id}

\definecolor{codegreen}{rgb}{0,0.6,0}
\definecolor{codegray}{rgb}{0.5,0.5,0.5}
\definecolor{codepurple}{rgb}{0.58,0,0.82}
\definecolor{backcolour}{rgb}{0.95,0.95,0.92}

\lstdefinestyle{mystyle}{
    backgroundcolor=\color{backcolour},   
    commentstyle=\color{codegreen},
    keywordstyle=\color{magenta},
    numberstyle=\tiny\color{codegray},
    stringstyle=\color{codepurple},
    basicstyle=\ttfamily\footnotesize,
    breakatwhitespace=false,         
    breaklines=true,                 
    captionpos=b,                    
    keepspaces=true,                 
    numbers=left,                    
    numbersep=5pt,                  
    showspaces=false,                
    showstringspaces=false,
    showtabs=false,                  
    tabsize=2
}

\usepackage[disable]{todonotes}

\usetikzlibrary{positioning}
  \def\gn#1#2{{$\href{http://groupnames.org/\#?#1}{#2}$}}
\def\gn#1#2{$#2$}  
\tikzset{sgplattice/.style={inner sep=1pt,norm/.style={red!50!blue},char/.style={red!80!black},
  lin/.style={black!50}},cnj/.style={black!50,yshift=-2.5pt,left=-1pt of #1,scale=0.5,fill=white}}

\usepackage{multicol}

\begin{document}

\begin{abstract}
    Global transfer systems are combinatorial objects that parametrize intermediate levels of structures and properties in globally equivariant homotopy theory. In this paper we study the lattice of global transfer systems, expanding beyond the case of abelian groups which was computed by the second author. The main result is a proof that the irreducible elements of the lattice of global transfer systems are strictly controlled by subgroups of groups which we refer to as \emph{endopure}. From the irreducible elements one can recover the whole lattice, and studying these endopure subgroups leads us to a wide range of new phenomena and computations.
\end{abstract}

\maketitle

\tableofcontents

\section{Introduction}

 Given a family $\fami$ of finite groups, a \emph{global $\fami$-transfer system}, as introduced in \cite{globalNinf}, is a collection of subgroup pairs $H \leq G \in \fami$ subject to certain compatibility conditions:
 
\begin{defintro}[\zcref{defi:gts}]
    A \emph{global $\mathcal{F}$-transfer system} $T$ is a binary relation $\leqslant_T$ on $\mathcal{F}$ such that:
        \begin{enumerate}
            \item The relation $\leqslant_T$ is a partial order that refines the subgroup relation on $\mathcal{F}$.
            \item The relation $\leqslant_T$ is closed under restriction along arbitrary homomorphisms. That is, for any $G,H,K \in \mathcal{F}$ and any homomorphism $\phi \colon G \to H$, if $K \leqslant_T H$, then $\phi^{-1}(K) \leqslant_T G$.
        \end{enumerate}
\end{defintro}

Global transfer systems originate in the setting of global homotopy theory which studies objects which have equivariant structures for all groups in a given family, and which are compatible with restrictions along all group homomorphisms. This strong compatibility requirement can be quire restrictive, in the sense that most equivariant $G$-spaces and $G$-spectra cannot be globally refined, but in turn the resulting additional structure can lead to strong results. See \cite{global_book} for an overview of the theory.

Global transfer systems are equivalent to \emph{atomic orbital} subcategories of the global orbit category, and as such they parametrize intermediate levels of global semiadditivity and stability by work of Cnossen, Lenz and Linskens \cite{MR4990070}. Global transfer systems are also equivalent to \emph{global $N_\infty$-operads} by \cite{globalNinf}. Both of these interpretations allow one to construct a \emph{transfer map} on the relevant homotopical object for each transfer pair $H \leq_T G$ in a global transfer system, which motivates their name.

For each group $G \in \fami$, one can restrict a global transfer system to obtain an equivariant $G$-transfer system. The classification of equivariant transfer systems is a very active and current research topic, which was started in \cite{BBR}. In \cite{abeliantransfers} the second author classified global transfer systems for abelian groups. In this paper we expand our interest to the non-abelian case.

Global transfer systems for a family $\fami$ form a complete lattice $\mathsf{Tr}^{\mathsf{Gl}}(\fami)$, with meet given by intersection, and join given by the global transfer system generated by the union. In fact, in the $G$-equivariant setting this join of transfer systems corresponds to the Boardman-Vogt tensor product of associated $N_\infty$-operads by \cite{stewart2025tensorproducts}. Well-behaved complete lattices can be recovered from their sets of join-irreducible and meet-irreducible elements, plus the induced relation between these. This is the theory of \emph{Formal Concept Analysis}, which interprets these irreducible elements as \emph{objects} and \emph{attributes}, and refers to the elements of the recovered lattice as \emph{formal concepts}, see \cite{FCAbook} for example for an introduction to the topic. We review the relevant theory in \zcref{subsec:fca}.

This characterization of a lattice in terms of its irreducible elements was employed by the first author and Spitz in \cite{balchinspitz} to study the lattice of equivariant $G$-transfer systems. In the equivariant setting, the irreducible elements correspond precisely to transfer systems generated by a single transfer pair $H \lneq K$, and the induced relation between these irreducible elements is easily understood. This leads to a powerful computational tool for exploring the collection of equivariant transfer systems for a fixed group $G$. 

The goal of the current paper is to port this computational power to the global case. In the global setting we can immediately see that any join-irreducible global transfer system can also be generated by a single transfer pair, as any global transfer system can be written as an infinite join of the global transfer systems generated by the individual transfers it contains. However, unlike in the equivariant case, a direct examination of all but the simplest of examples shows that a single transfer does not usually generate a join-irreducible element. This leads to the question of determining which transfers generate join- and meet-irreducible global transfer systems, which we shall see is completely answered by a special class of subgroups that we will now describe.

Consider a subgroup pair $H \lneq G$. The global transfer system generated by it consists of all transfers that can be written as a composition of restrictions of $H \lneq G$. Thus, if we can write it as a non-trivial join, we are also able to write the pair $H \lneq G$ as a composition of smaller pairs, which are themselves restrictions of $H \lneq G$. Consider the topmost pair in this composition, it must be a restriction along some endomorphism $\phi \colon G \to G$, so that $\phi^{-1}(H)$ sits strictly between $H$ and $G$. Therefore, if we can guarantee that there exists no such $\phi$ for $H \lneq G$, we will obtain that the global transfer system generated by $H \lneq G$ is join-irreducible. For example, any subgroup pair of prime index is of this form.

One issue with this characterization is that different transfers can generate the same global transfer system. This is the case, for example, for $1 \lneq C_2$ and $C_3 \lneq C_6$. To remedy this, we use the observation that if there exists a non-injective $\phi \colon G \to G$ with $\phi^{-1}(H)=H$, then we can construct a smaller transfer that will generate the same global transfer system as $H \lneq G$. For this reason, we will disallow the existence of such $\phi$, leading to the following key definition:

\begin{defintro}[\zcref{defi:endopure}]
    Let $H \lneq G$ be a proper subgroup. Then $H$ is \emph{endopure} in $G$ if there are no non-injective endomorphisms $\phi$ of $G$ with $H \leq \phi^{-1}(H) \lneq G$. 
\end{defintro}

These endopure subgroups behave as building blocks under restriction along homomorphisms of the poset of all finite groups and subgroups, and as far as we can tell this is a novel group theoretic notion. All proper subgroups of simple groups are endopure, but it turns out that there are many more examples than these, some of which are unexpected. In \zcref{sec:endopure} and \zcref{sec:examples} we examine many of these. We highlight one particular result:

\begin{thmintro}[\zcref{prop:abendo}]
    Let $G$ be a finite abelian group. Then $H \lneq G$ is endopure if and only if $G = C_{p^n}$ for some prime $p$, and $H = C_{p^{n-1}}$.
\end{thmintro}

We prove in \zcref{sec:irreducibles} that endopure transfers inside the family $\fami$ generate and uniquely determine all join-irreducible global transfer systems. The classification of meet-irreducible global transfer systems requires comparatively more work. We prove that under a different construction, endopure pairs also generate and uniquely determine all the meet-irreducible global transfer systems. This, combined with the fundamental theorem of Formal Concept Analysis (see \zcref{thm:FTFCA2}) yields our main theorem:

\begin{thmintro}[\zcref{thm:summary}]
    Let $\mathcal{F}$ be a family of finite groups and $\mathsf{EP}(\fami)$ the collection of endopure transfers in $\fami$. Then there is a canonical isomorphism
    \[\varepsilon_\fami \colon \mathsf{Tr}^\mathsf{Gl}(\fami) \to \conlat(\mathsf{EP}(\fami), \mathsf{EP}(\fami), {\leq_{\mathsf{EP}}})\]
    from the complete lattice of $\fami$-global transfer systems to the lattice of formal concepts in the formal context of endopure pairs in $\fami$. This isomorphism sends $T$ to the formal concept given by the two subsets \[(\{(L\leq K) \in \mathsf{EP}(\fami) \mid (L \leq K) \in T\},\{(H\leq G) \in \mathsf{EP}(\fami) \mid  (H' \leq G) \notin T \; \text{for every} \; H \leq H' \lneq G \}).\]
    
    The relation $\leq_{\mathsf{EP}}$ in $\mathsf{EP}(\fami) \times \mathsf{EP}(\fami)$ is given by $(L \leq K) \leq_{\mathsf{EP}} (H \leq G)$ if there are no homomorphisms $\phi \colon G \to K$ with $H \leq \phi^{-1}(L) \lneq G$.
\end{thmintro}

This provides us with not only a conceptual understanding of how global transfer systems are constructed, but also a tangible computational tool. Given a family $\fami$, the main task is to classify the endopure transfers appearing in $\fami$. We provide \texttt{GAP} code for this  in the \hyperlink{appendix}{Appendix} so that the interested reader can perform their own computations.

In \zcref{sec:examples} we apply the machinery to obtain new results about global transfer systems in several families. Before this paper, only abelian examples had been explicitly computed, in \cite{abeliantransfers}. Of particular intrigue, as always, is the family of extraspecial $p$-groups:

\begin{thmintro}[\zcref{thm:esfinal}] Let $p$ be prime, and denote by $\mathcal{E}_p$ the family consisting of all extraspecial $p$-groups and their subgroups. Then 
    \[
        |\mathsf{Tr}^\mathsf{Gl}(\mathcal{E}_p)| = \begin{cases}
            5 & p=2 \\
            \infty & p \text{ odd}.
        \end{cases}
    \]
\end{thmintro}

Using the computational tools facilitated by the theory in this paper, we are able to compute statistics regarding the most omnipresent families of groups, $\mathcal{G}_n$, containing the finite groups of order $\leq n$. Taking the limit over all these families (in an explicit sense discussed in \zcref{subsec:filts}) we can conjecture facts about the family $\mathcal{G}$ of \textit{all} finite groups. For a family $\fami$, the construction of the formal context $(\mathsf{EP}(\fami), \mathsf{EP}(\fami), {\leq_{\mathsf{EP}}})$ gives rise to a binary matrix $C(\fami)$ whose entry corresponding to $((L \leq K), (H \leq G))$ in $\mathsf{EP}(\fami) \times \mathsf{EP}(\fami)$ is $1$ if $(L \leq K) \leq_{\mathsf{EP}} (H \leq G)$, and $0$ otherwise. The density $\delta(\fami)$ is the number of ones in $C(\fami)$ divided by $\lvert \mathsf{EP}(\fami) \rvert^2$. This density allows us to discuss bounds for the size of the corresponding lattice. In many situations of interest, the limiting density over some filtered family is 1 (see \cite{balchinspitz} and \zcref{subsec:ab}). Our computational evidence leads us to the following, surprising conjecture:

\begin{conjintro}[\zcref{subsec:allgroups}]
Let $\mathcal{G}_n$ be the family of finite groups of order $\leq n$. Then $0.5 < \lim_n \delta(\mathcal{G}_n) < 1$.
\end{conjintro}

\subsection*{Acknowledgements}

MB is supported by the European Research Council (ERC) under Horizon Europe (grant No. 101042990). SB and MB would like to thank the Max Planck Institute for Mathematics in Bonn for its hospitality and financial support. No generative AI was used for the content or writing of this paper.

\section{Recollections}

\subsection{Global transfer systems}

We begin with a brief recap of the theory of global transfer systems from \cite{abeliantransfers}. For an introduction to the wider concepts surrounding global homotopy theory, the reader may consult \cite{global_book} or \cite{LNP} (although note that here we will work with arbitrary families of finite groups, as opposed to global families of compact Lie groups).

\begin{defi}
    A \emph{family} of finite groups $\mathcal{F}$ is a non-empty collection of finite groups closed under subgroups and isomorphisms.
\end{defi}

\begin{ex}
    Let $G$ be a finite group. Then $\Sub(G)$, the collection of all groups isomorphic to a subgroup of $G$, is a family.
\end{ex}

\begin{defi}
\label{defi:gts}
    Let $\mathcal{F}$ be a family. A \emph{global $\mathcal{F}$-transfer system} $T$ is a binary relation $\leqslant_T$ on $\mathcal{F}$ such that:
        \begin{enumerate}
            \item The relation $\leqslant_T$ is a partial order that refines the subgroup relation on $\mathcal{F}$.
            \item The relation $\leqslant_T$ is closed under pullback/restriction along arbitrary homomorphisms. That is, for any $G,H,K \in \mathcal{F}$ and any homomorphism $\phi \colon G \to H$, if $K \leqslant_T H$, then $\phi^{-1}(K) \leqslant_T G$.
        \end{enumerate}
\end{defi}

\begin{defi}
    Let $\mathcal{F}$ be a family. We will call a pair $H \leqslant G$ with $H, G \in \mathcal{F}$ a \emph{transfer}, and we will often denote such a transfer with $(H \leqslant G)$ for clarity. We will write $\mathbb{I}(\mathcal{F})$ for the set of transfers in $\fami$. For a global $\mathcal{F}$-transfer system $T$, we will write $(H \leqslant G) \in T$ if $H \leqslant_T G$. In this sense, a global transfer system is a set of transfers satisfying certain conditions, which justifies the name.

    Given two transfers, we will say that $(H \leq G)$ is a \emph{restriction} of $(H' \leq G')$ if there is a homomorphism $\phi \colon G \to G'$ with $\phi^{-1}(H')=H$. Given $G,H,K \in \mathcal{F}$ with $K \leq H \leq G$ we will say that $(K \leq G)$ is a \emph{composition} of $(K \leq H)$ and $(H \leq G)$. Thus, a global $\fami$-transfer system is a set of transfers in $\fami$ closed under composition and restriction.

    We will say that two transfers $(H \leq G)$ and $(H' \leq G')$ are isomorphic if there exists an isomorphism $\phi \colon G \to G'$ with $\phi(H) = H'$. Given a global $\mathcal{F}$-transfer system $T$, it is clear by the restriction condition of \zcref{defi:gts} that if two transfers are isomorphic, then one belongs to $T$ if and only if the other does. This is not true if one merely knows that $G \cong G'$ and $H \cong H'$ abstractly, as can be seen by looking at different copies of $C_2$ within the dihedral group $D_6$.
\end{defi}

\begin{rem}
    We denote the collection of all global $\mathcal{F}$-transfer systems by $\mathsf{Tr}^\mathsf{Gl}(\mathcal{F})$. This is a complete lattice where the order is given by containment, and meet is given by intersection. The join of a set of global transfer systems is the smallest global transfer system that contains their union, which is more concretely the transitive closure of their union (i.e., its closure under composition of transfers).
\end{rem}

\begin{rem}
    Let $G$ be a finite group. We will often write $\mathsf{Tr}^\mathsf{Gl}(G)$ for $\mathsf{Tr}^\mathsf{Gl}(\Sub(G))$. Global transfer systems for the family $\Sub(G)$ are in particular equivariant transfer systems for $G$ in the usual sense (see for example \cite[\S 2]{BBR}) with the added requirement that they are closed under restriction along arbitrary homomorphisms instead of just under conjugation and restriction to subgroups. This observation provides a lattice embedding $\mathsf{Tr}^\mathsf{Gl}(G) \hookrightarrow \mathsf{Tr}(G)$, see \cite[\S7]{globalNinf}. 
\end{rem}

\begin{ex}
    Let $G = C_{p^n}$ for $n \geqslant 1$ and $p$ prime. Then by \cite[Theorem~25]{BBR} we have $\mathsf{Tr}(G) \cong \mathcal{A}_{n+1}$, the $(n+1)$-associahedron, and thus $|\mathsf{Tr}(G)|$ is given by the $(n+1)$st Catalan number. In stark contrast, we have that $\mathsf{Tr}^\mathsf{Gl}(G) \cong [n]$, see \cite[Example 7.3]{globalNinf}.
\end{ex}

We will later need the following simple observation regarding restriction of transfers.

\begin{lemm}
\label{lemm:res-index}
    If $(H \leq G)$ is a restriction of $(H' \leq G')$ we have $[G' \colon H'] \geq [G \colon H]$.
\end{lemm}
\begin{proof}
    There exists a homomorphism $\phi \colon G \to G'$ with $\phi^{-1}(H')=H$. For $g, g' \in G$, if $\phi(g)^{-1} \phi(g') \in H'$ then $g^{-1} g' \in \phi^{-1}(H)$. Thus the function that $\phi$ induces on cosets $G/H \to G'/H'$ is injective.
\end{proof}

\subsection{Formal concept analysis}\label{subsec:fca}

We now briefly recall the theory of formal concept analysis that we will need for our application. We follow the exposition of \cite{balchinspitz}.

\begin{defi}
    A \emph{formal context} is an ordered triple $(X, Y, R)$ where $X$ and $Y$ are non-empty sets and $R \subset X \times Y$ is a relation. Elements $x \in X$ are called \emph{objects}, and elements $y \in Y$ are called \emph{attributes}.
\end{defi}

\begin{rem}
    We will usually represent a formal context as an $|X| \times |Y|$ binary matrix denoted $C$, with $C(x,y) = 1$ if $xRy$ and 0 otherwise.
\end{rem}

\begin{defi}
    Let $(X, Y, R)$ be a formal context. Define two operators $(-)^\uparrow \colon \mathcal{P}(X) \to \mathcal{P}(Y)$ and $(-)^\downarrow \colon \mathcal{P}(Y) \to \mathcal{P}(X)$ as
    \[
        A^\uparrow = \{y \in Y \mid \text{for each } x \in A,\, (x,y) \in R\} \quad \text{and} \quad 
        B^\downarrow = \{x \in X \mid \text{for each } y \in B,\, (x,y) \in R\}.
    \]
\end{defi}

\begin{defi}\label{defn:ofconcept}
    Let $(X, Y, R)$ be a formal context. A \emph{formal concept} in $(X,Y,R)$ is a pair $A \subset X$, $B \subset Y$ such that $A^\uparrow = B$ and $B^\downarrow = A$.

    Denote by $\conlat(X, Y, R)$ the collection of all formal concepts of $(X, Y, R)$. It is naturally partially ordered by $(A_1,B_1) \leq (A_2,B_2) \iff A_1 \subset A_2$ (equivalently if and only if $B_1 \supseteq B_2$).
\end{defi}

\begin{thm}[Fundamental theorem of Formal Concept Analysis, Part I \cite{Wille1982Restructuring}]\label{thm:FTFCA1}
    Let $(X, Y, R)$ be a formal context. Then $\conlat(X, Y, R)$ is a complete lattice with meet and join given by
    \[
        \bigwedge_{j \in J} (A_j,B_j) = (\bigcap_{j \in J} A_j, (\bigcup_{j \in J}B_j)^{\downarrow \uparrow}) \quad \text{and} \quad
        \bigvee_{j \in J} (A_j,B_j) = ((\bigcup_{j \in J}A_j)^{\uparrow \downarrow}, \bigcap_{j \in J} B_j)
    \]
\end{thm}

\begin{defi}
    For a context $(X,Y,R)$, an object $x \in X$ is reducible if there is a (possibly empty) set $S \subset X \setminus \{ x\}$ such that $\{x\}^\uparrow = S^\uparrow$, and dually for $y \in Y$. The context $(X,Y,R)$ is \emph{reduced} if no object of $X$ or $Y$ is reducible. In other words, a context is reduced if no row (resp., column) can be obtained as a multiplication of a set of different rows (resp., columns) in a matrix representation of $R$.
\end{defi}

Given any context, one can \emph{reduce} it by iteratively removing rows and columns until a reduced context is obtained. We shall call this the \emph{reduction} of the context. One sees that the concept lattices for a context and its reduction are isomorphic, see \cite[Proposition 14]{FCAbook}.

Recall that for $L$ a lattice, an element $x \in L$ is said to be \emph{join-irreducible} (resp., \emph{meet-irreducible}) if, for all $X \subset L$, $x = \bigvee X$ implies $x \in X$ (resp., $x = \bigwedge X$ implies $x \in X$). We denote by $J(L)$ the set of join-irreducible elements of $L$, and by $M(L)$ the collection of meet-irreducible elements of $L$. 

\begin{thm}[The Fundamental Theorem of Formal Concept Analysis, Part II \cite{Wille1982Restructuring}]\label{thm:FTFCA2}
    Let $L$ be a complete lattice such that each element of $L$ can be expressed both as a join of join-irreducible elements and as a meet of meet-irreducible elements. Then there is a canonical isomorphism $L \to \conlat(J(L), M(L), {\leq})$ given by
    \[z \mapsto (\{x \in J(L) \mid x \leq z\},\{y \in M(L) \mid z \leq y\}).\]

    Moreover, the context $(J(L), M(L), {\leq})$ is reduced, and is the unique (up to isomorphism) reduced context with concept lattice isomorphic to $L$.
\end{thm}

\zcref{thm:FTFCA2} thus allows us to reconstruct a lattice $L$ purely from the data of the join- and meet-irreducible elements and the inclusions between them. This point of view was key in \cite[Theorem~A]{balchinspitz} where it was proved that for the lattice $\mathsf{Tr}(G)$ of equivariant transfer systems for a group $G$, the irreducible elements, along with the inclusion relation, admits a simple description. 

At this stage, the point of the current paper is clear, we wish to apply the Fundamental Theorem of Formal Concept Analysis to the \textit{global} setting.

\begin{rem}
    Unlike in \cite{balchinspitz}, our lattices need not be finite. As such, we will need to verify that the lattices satisfy the technical condition of \zcref{thm:FTFCA2}, which we will do in \zcref{subsec:condition}.
\end{rem}

\section{Endopurity of subgroups}
\label{sec:endopure}

\subsection{First definitions and examples}

In our hunt for a characterisation of the irreducible global transfer systems in \zcref{sec:irreducibles}, we arrive to a condition on subgroups of finite groups. We dedicate this section to an initial exploration into this seemingly new notion. We begin with the definition:

\begin{defi}\label{defi:endopure}
    Let $G$ be a finite group, and let $H \lneq G$. Then $H$ is \emph{endopure} in $G$ if there are no non-injective endomorphisms $\phi$ of $G$ with $H \leq \phi^{-1}(H) \lneq G$. 
    
    We denote by $\mathsf{EP}(G)$ the collection of endopure subgroups of $G$, where we identify $H \sim K$ whenever there exists $\alpha \in \mathrm{Aut}(G)$ such that $\alpha(H) = K$.
    \end{defi}

\begin{rem}\label{rem:reform}
    \zcref{defi:endopure} can equivalently be restated as saying that $H$ is endopure if for any non-injective $\phi \in \mathrm{End}(G)$ with $\phi(H) \leq H$, we have that $\phi(G) \leq H$. 
\end{rem}

\begin{defi}\label{defi:endocore}
    Let $G$ be a finite group. Then the \emph{endocore} of $G$ is the following subgroup:
    \[
    \Omega(G) \coloneq \la \phi(G) \mid \phi \in \mathrm{End}(G),\,\phi \text{ non-injective} \ra \leq G.
    \]
\end{defi} 

\begin{lemm}\label{lem:endocore}
    Let $\Omega(G) \leq H \lneq G$. Then $H$ is endopure.
\end{lemm}

\begin{proof}
    This immediately follows from the equivalent description of endopurity of \zcref{rem:reform}.
\end{proof}

\begin{rem}
    The converse to \zcref{lem:endocore} is false. For an explicit example take $G = S_3$. Then for each subgroup $K = \la \tau \ra$ for $\tau$ a transposition there is an endomorphism $\phi$ with $\phi(S_3) = K$. Since $S_3$ is generated by transpositions, $\Omega(S_3)=S_3$. However, $K = \la \tau \ra$ is endopure in $S_3$. We will see a family where the converse does hold in \zcref{thm:endoiffcore}.
\end{rem}

\begin{ex}\label{ex:cpn}
    Let $G = C_{p^n}$. Then the only endopure subgroup is $C_{p^{n-1}}$. Indeed, suppose $G = \la g \ra$, and consider the non-injective endomorphism $\phi \colon G \to G$ given by $\phi(g) = g^p$. Then for $H_i \coloneq \la g^{p^{n-i}} \ra \leq G$, we have
    \[
    \phi^{-1}(H_i) = \begin{cases} H_{i+1} & i < n-1 \\G & i=n-1.\end{cases}
    \]
    For $i < n-1$ we as such get $H_i \leq \phi^{-1}(H_i) = H_{i+1} \lneq G$, and as such $H_i$ is not endopure. When $i=n-1$, one observes that $\Omega(G) = H_{n-1}$, and as such, it is endopure.
    
    In fact, the main result of this section, \zcref{prop:abendo}, is that in the abelian case this is \textit{all} that can happen, i.e., for any other abelian group $G$, $\mathsf{EP}(G) = \varnothing$.
\end{ex}

\begin{ex}\label{ex:SL2F5}
    It is clear that if $G$ is a simple group, then any proper subgroup of $G$ is endopure as $\Omega(G) = 1$. The converse, however, need not be true. For example, the trivial subgroup 1 is endopure if and only if $G$ has no non-trivial non-injective endomorphisms. An example of such a group which is not simple is $\mathrm{SL}_2(\mathbb{F}_p)$ for $p \geq 5$. We further note that any such group must be perfect.
\end{ex}

\begin{rem}
\label{rem:s6}
    We give an explicit example to showcase that identifying endopure subgroups of $G$ up to automorphism of $G$, as opposed to abstract isomorphism, makes a difference.

    Let $G = S_6$. Recall that $S_6$ is unique among the symmetric groups, as it possesses an outer automorphism, which we will denote $\sigma$. Consider $V_4=C_2^2$. There are five conjugacy classes of $V_4$ in $S_6$:
    \begin{enumerate}[align = left]
        \item[Type I:] $\langle (i_1i_2)(i_5i_6), (i_3i_4)(i_5i_6)  \rangle$,
        \item[Type II:] $\langle (i_1i_3)(i_2i_4), (i_1i_2)(i_3i_4) \rangle$,
        \item[Type III:] $\langle (i_1i_2)(i_5i_6), (i_3i_4) \rangle$,
        \item[Type IV:] $\langle (i_1i_2)(i_3i_4), (i_3i_4) \rangle$,
        \item[Type V:] $\langle (i_1i_3)(i_2i_4), (i_1i_2)(i_3i_4)(i_5i_6) \rangle$.
    \end{enumerate}
    Among these, Types III, IV, and V are endopure. The outer automorphism $\sigma$ swaps the conjugacy classes of Type IV and V, while fixing those of Type III. As such, there are two distinct copies of $V_4$ in $\mathsf{EP}(S_6)$. If we were to just take the subgroups up to abstract isomorphism, we would only have one copy of $V_4$. As such we would not correctly identify enough join- and meet-irreducible elements of the lattice of global transfer systems for $\Sub(S_6)$, and thus miss some global transfer systems for this family.
\end{rem}

\begin{rem}
    From explicitly computing small examples by hand, the authors initially conjectured that the required condition would be that $H$ is ``good'' if there are no proper normal subgroups of $G$ containing $H$. However when trying to prove results regarding the irreducible elements of $\mathsf{Tr}^\mathsf{Gl}(\fami)$ the role of endopurity became clear, as described in the introduction. This other condition is incomparable to endopurity:
    \begin{itemize}
        \item $C_2 \leq \mathrm{SL}_2(\mathbb{F}_5)$ is a normal subgroup, and thus the trivial subgroup is not ``good''. However, we have seen in \zcref{ex:SL2F5} that the trivial subgroup is endopure.
         \item Let $G = C_3 \ltimes S_3$ where the action is via $S_3/C_3 = C_2$. Let $H = C_2$. Then there are no normal proper subgroups of $G$ that contain $H$, and thus it is ``good'' in the above sense. However, $H \leq G$ is not endopure as we can choose a homomorphism $\phi$ of the form $G \to G/C_3 \cong S_3 \hookrightarrow G$ such that the preimage of our chosen $H$ is a copy of $S_3$ that contains it.
    \end{itemize}
    
    For the convenience of the reader we display the subgroup lattices of both the groups discussed above in \zcref{fig:subgroups}.
\end{rem}

\begin{figure}[h!]
\centering
    \begin{tikzpicture}[scale=1.0,sgplattice]
     
  \node[char] at (2,0) (1) {\gn{C1}{1}};
  \node[char] at (2,0.803) (2) {\gn{C2}{C_2}};
  \node at (3.88,1.76) (3) {\gn{C3}{C_3}};
  \node at (0.125,1.76) (4) {\gn{C5}{C_5}};
  \node at (2,1.76) (5) {\gn{C4}{C_4}};
  \node at (3.88,2.97) (6) {\gn{C6}{C_6}};
  \node at (0.125,2.97) (7) {\gn{C10}{C_{10}}};
  \node at (2,2.97) (8) {\gn{Q8}{Q_8}};
  \node at (3.88,4.19) (9) {\gn{Dic3}{{\rm Dic}_3}};
  \node at (0.125,4.19) (10) {\gn{Dic5}{{\rm Dic}_5}};
  \node at (2,4.19) (11) {\gn{SL(2,3)}{{\rm SL}_2({\mathbb F}_3)}};
  \node[char] at (2,5.14) (12) {\gn{SL(2,5)}{{\rm SL}_2({\mathbb F}_5)}};
  \draw[lin] (1)--(2) (1)--(3) (1)--(4) (2)--(5) (2)--(6) (3)--(6) (2)--(7)
     (4)--(7) (5)--(8) (5)--(9) (6)--(9) (5)--(10) (7)--(10) (6)--(11)
     (8)--(11) (9)--(12) (10)--(12) (11)--(12);
  \node[cnj=3] {10};
  \node[cnj=4] {6};
  \node[cnj=5] {15};
  \node[cnj=6] {10};
  \node[cnj=7] {6};
  \node[cnj=8] {5};
  \node[cnj=9] {10};
  \node[cnj=10] {6};
  \node[cnj=11] {5};
\end{tikzpicture}
\qquad
\begin{tikzpicture}[scale=1.0,sgplattice, xscale=0.75]
\node at (0.125,-1) (dud) { };
  \node[char] at (4.25,0) (1) {\gn{C1}{1}};
  \node at (2.12,1.22) (2) {\gn{C2}{C_2}};
  \node[char] at (0.125,1.22) (3) {\gn{C3}{C_3}};
  \node[char] at (8.38,1.22) (4) {\gn{C3}{C_3}};
  \node[char] at (6.38,1.22) (5) {\gn{C3}{C_3}};
  \node[char] at (4.25,1.22) (6) {\gn{C3}{C_3}};
  \node at (0.125,2.87) (7) {\gn{S3}{S_3}};
  \node at (8.38,2.87) (8) {\gn{S3}{S_3}};
  \node at (6.38,2.87) (9) {\gn{S3}{S_3}};
  \node at (2.12,2.87) (10) {\gn{S3}{S_3}};
  \node[char] at (4.25,2.87) (11) {\gn{C3^2}{C_3^2}};
  \node[char] at (4.25,4.09) (12) {\gn{C3:S3}{C_3{\rtimes}S_3}};
  \draw[lin] (1)--(2) (1)--(3) (1)--(4) (1)--(5) (1)--(6) (2)--(7) (3)--(7)
     (2)--(8) (4)--(8) (2)--(9) (5)--(9) (2)--(10) (6)--(10) (3)--(11)
     (4)--(11) (5)--(11) (6)--(11) (7)--(12) (8)--(12) (9)--(12) (10)--(12)
     (11)--(12);
  \node[cnj=2] {9};
  \node[cnj=7] {3};
  \node[cnj=8] {3};
  \node[cnj=9] {3};
  \node[cnj=10] {3};
\end{tikzpicture}
\caption{The subgroup lattices of $\mathrm{SL}_2(\mathbb{F}_5)$ and $C_3 \ltimes S_3$. Normal subgroups are highlighted in red.}\label{fig:subgroups}
\end{figure}

\subsection{Endopurity for abelian groups}

We will now prove the aforementioned result regarding the lack of endopure subgroups among abelian groups.

\begin{thm}
    \label{prop:abendo}
    Let $G$ be a finite abelian group. Then $H \lneq G$ is endopure if and only if $G = C_{p^n}$ for some prime $p$, and $H = C_{p^{n-1}}$ 
\end{thm}

\begin{proof}
   From the discussion of \zcref{ex:cpn}, it suffices to show that if $G$ is not of this form then it has no endopure subgroups.

   We will begin by assuming that $G$ is not a $p$-group. Then $|G|$ has at least two prime factors, and as such it has at least two distinct Sylow subgroups and we write  $G = \prod_{p \mid |G|}G_p$ for the decomposition of $G$ into Sylow subgroups.

   Let $H \lneq G$. As $G$ is finite abelian we can equally split $H$ among the Sylow subgroups as $H = \prod_{p \mid |G|}H_p$ with each $H_p \leq G_p$. Since $H \lneq G$ we can find at least one $p$ such that $H_p \lneq G_p$. We then let $\phi \in \mathrm{End}(G)$ be the projection onto $G_p$ composed with the natural inclusion $G_p \hookrightarrow G$. As we have assumed that $G$ has at least one other non-trivial Sylow subgroup, $\phi$ is non-injective and clearly $H \leq \phi^{-1}(H) \lneq G$ as required.

   In light of the above, we now reduce to the case that $G$ is an abelian $p$-group. We consider two cases:

   \begin{enumerate}
       \item $pG \not\leq H$. Here we simply use the endomorphism $\phi(g) = pg$ which is non-injective as $G$ is a $p$-group. We have $\phi(H) = pH \leq H$, but $\phi(G)=pG \not\leq H$. Thus $H \leq \phi^{-1}(H) \lneq G$. That is, $H$ is not endopure.
       \item $pG \leq H$. This case is somewhat more involved, and we shall explicitly construct a non-injective endomorphism of $G$ using the generators $g_j$. The group $G$ necessarily has more than one cyclic factor, in particular we may write $G = \langle g_1 \rangle \times \cdots \times \langle g_r \rangle$ with $r \geq 2$ and $|g_i| = p^{n_i}$, and an arbitrary $g \in G$ can be written as $g = a_1g_1 + \cdots a_r g_r$. 
       
       Since we have assumed $pG \leq H$ and $H \lneq G$, we have $H/pG \lneq G/pG$, and there exists some non-zero homomorphism $\theta \colon G/pG \to C_p$ with $\theta(H/pH) = 0$. We want to upgrade this homomorphism to a homomorphism $\phi \colon G \to G$. For this, define $c_j$ to be $\theta(g_j+pG)$, choose $i$ such that $c_i \neq 0$ (which exists because $\theta$ is non-zero), and set $\phi(g_j) = c_j g_i$, extending linearly to $G$. 
       
       Since $\theta(H/pH) = 0$, for every element $h = a_1g_1 + \cdots a_r g_r \in H$ we have $c_1a_1+ \cdots + c_r a_r \equiv 0$ modulo $p$. Thus $\phi(h) = (\sum_j a_jc_j )g_i \in pG \leq H$. On the other hand $\phi(g_i) = c_ig_i \notin H$ as $g_i \notin H$, so $\phi(G) \not\leq H$. Lastly $\phi$ is non-injective as its image lies in $\langle g_i \rangle$, and thus $\phi$ shows that $H \leq G$ is not endopure.\qedhere
       
   \end{enumerate}
\end{proof}

While the abelian case can be completely resolved, the non-abelian case of direct products is more intricate. Naively, one may expect that any direct product group contains no endopure subgroups, but this is not the case. For $H_1 \leq G_1, H_2\leq G_2$ we always have that $H_1 \times H_2 \leq G_1 \times G_2$ is not endopure, which follows from the argument given at the start of the proof above. However, outside of this restrictive condition, we can have endopure subgroups of direct products. Below we provide the minimal (in terms of order of $G$) such example.

\begin{ex}\label{ex:c4xq8}
    Let $G = C_4 \times Q_8$. Then it is possible to check that $G$ contains two endopure subgroups (see the \hyperlink{appendix}{Appendix}) namely $H \cong C_4 \times C_2$ and $K \cong C_4 \rtimes C_4$. Writing $C_4 \cong \la a \ra$ and $Q_8 = \{1,i,j,k,-1,-i,-j,-k\}$ we can explicitly identify
    \[
    H =  \{(1,1),(1,-1),(a,i),(a,-i),(a^2,1),(a^2,-1),(a^3,i),(a^3,-i)\}
    \]
    We know that $H$ must be a diagonal-type subgroup, as otherwise it would not be endopure, and indeed its projections are both isomorphic to $C_4$, and as such if $H$ were of the form $H_1 \times H_2$ for $H_1 \leq C_4$ and $H_2 \leq Q_8$ it would be forced to have order 16.
\end{ex}

\subsection{Endopure transfers}

We finish this section by extending the notion of endopurity to families.

\begin{defi}
    Let $\mathcal{F}$ be a family. A transfer $(K \leq H) \in \mathbb{I}(\mathcal{F})$ is an \emph{endopure transfer} if $K \in \mathsf{EP}(H)$. We define
    \[\mathsf{EP}(\mathcal{F}) = \{ (K \leq H) \in \mathbb{I}(\mathcal{F}) \mid K \in \mathsf{EP}(H) \} / \cong\]
    to be the set of endopure transfers of $\mathcal{F}$, up to isomorphism of transfers.
\end{defi}

\begin{ex}\label{ex:epofsl2}
Let $G = \mathrm{SL}_2(\mathbb{F}_5)$, with its subgroup lattice shown in \zcref{fig:subgroups}. As a showcase we will  classify all endopure transfers of $\mathcal{F} = \Sub(G)$. It immediately follows from \zcref{ex:SL2F5} that $\Omega(\mathrm{SL}_2(\mathbb{F}_5)) = 1$, and hence every proper subgroup of $\mathrm{SL}_2(\mathbb{F}_5)$ is endopure, giving us endopure transfers:
\begin{multicols}{4}
    \begin{enumerate}[leftmargin=*,series=qa]
        \item $1 \leq \mathrm{SL}_2(\mathbb{F}_5)$,
        \item $C_2 \leq \mathrm{SL}_2(\mathbb{F}_5)$,
        \item $C_5 \leq \mathrm{SL}_2(\mathbb{F}_5)$,
        \item $C_4 \leq \mathrm{SL}_2(\mathbb{F}_5)$,
        \item $C_3 \leq \mathrm{SL}_2(\mathbb{F}_5)$,
        \item $C_{10} \leq \mathrm{SL}_2(\mathbb{F}_5)$,
        \item $Q_8 \leq \mathrm{SL}_2(\mathbb{F}_5)$,
        \item $C_6 \leq \mathrm{SL}_2(\mathbb{F}_5)$,
        \item $\mathrm{Dic}_5 \leq \mathrm{SL}_2(\mathbb{F}_5)$,
        \item $\mathrm{SL}_2(\mathbb{F}_3) \leq \mathrm{SL}_2(\mathbb{F}_5)$,
        \item $\mathrm{Dic}_3 \leq \mathrm{SL}_2(\mathbb{F}_5)$.
    \end{enumerate}
\end{multicols}
The following are all the endopure transfers where the larger subgroup is not $G$:
\begin{multicols}{4}
    \begin{enumerate}[resume=qa,leftmargin=*,align=left]
        \item $1 \leq C_2$,
        \item $1 \leq C_3$,
        \item $C_2 \leq C_4$,
        \item $1 \leq C_5$,
        \item $C_2 \leq Q_8$,
        \item $C_4 \leq Q_8$,
        \item $C_4 \leq \mathrm{Dic}_3$,
        \item $C_4 \leq \mathrm{Dic}_5$,
        \item $C_3 \leq \mathrm{SL}_2(\mathbb{F}_3)$,
        \item $C_6 \leq \mathrm{SL}_2(\mathbb{F}_3)$.
    \end{enumerate}
\end{multicols}

In particular we conclude that $|\mathsf{EP}(\Sub(\mathrm{SL}_2(\mathbb{F}_5)))| = 21$.
\end{ex}

\section{Irreducible global transfer systems correspond to endopure transfers}
\label{sec:irreducibles}

The goal of this section is to study the join- and meet-irreducible elements of the lattice of global transfer systems for a family $\mathcal{F}$. We will show that each of these is determined by a single transfer, and that in fact these irreducible elements correspond to endopure transfers.

After introducing two constructions for global transfer systems determined from a single transfer, we prove a series of auxiliary technical results, which culminate in the following five facts about the lattice of global $\mathcal{F}$-transfer systems:

\begin{enumerate}
    \item Any global transfer system can be written as a join of join-irreducible global transfer systems (\zcref{prop:join-condition}).
    \item Any global transfer system can be written as a meet of meet-irreducible global transfer systems (\zcref{prop:meet-condition}).
    \item Join-irreducible global transfer systems correspond to endopure transfers (\zcref{prop:join-bijection}).
    \item Meet-irreducible global transfer systems correspond to endopure transfers (\zcref{prop:meet-bijection}).
    \item Whether a join-irreducible is contained in a meet-irreducible depends on an explicit condition on the corresponding endopure transfers (\zcref{prop:containment-condition}).
\end{enumerate}

Throughout this section we will fix an implicit family $\mathcal{F}$, and generally omit it from the notation.

\subsection{Generation and anti-generation of global transfer systems}

Just as in the equivariant case, it is possible to form the smallest global transfer system generated by a set of transfers $S \subset \mathbb{I}(\mathcal{F})$ using a variant of Rubin's algorithm, introduced in \cite[Theorem A.2]{rubin}. In the global case however conjugation is taken care of by restriction along inner automorphisms, so there are only two steps to this algorithm, instead of the three steps for equivariant transfer systems.

\begin{constr}
    We will denote by $\overline{\la S \ra}$ the global transfer system generated by $S$, the smallest global transfer system that contains $S$. It is the transitive closure of the set of restrictions of transfers in $S$. Isomorphic single transfers generate the same global transfer system, as they are mutual restrictions.
\end{constr}

In \cite{balchinspitz}, \textit{co}transfer systems were considered when discussing the meet irreducible elements of the lattice. However, the notion of a cotransfer system cannot be straightforwardly generalized to the global case. Here we will instead describe a construction of the transfer system \textit{antigenerated} by a set of transfers $S$ (although this is not the most appropriate name, see \zcref{rem:antigeneration-example}). As motivation for said construction, consider the following lemma pertaining to global transfer systems generated by a single transfer.

\begin{lemm}
\label{lemm:main-property-generate}
    If $(H \leq G) \in \overline{\la H' \leq G' \ra}$ then there exists a homomorphism $\phi \colon G \to G'$ with ${H \leq \phi^{-1}(H') \lneq G}$.
\end{lemm}
\begin{proof}
    If the transfer $(H \leq G)$ is in $\overline{\la H' \leq G' \ra}$, then it is a composition of restrictions of $(H' \leq G')$. Let $(H_0 \lneq G)$ be the top non-trivial transfer of this composition, so that $H \leq H_0 \lneq G$. Then as $(H_0 \lneq G)$ is a restriction of $(H' \leq G')$, there exists a homomorphism $\phi \colon G \to G'$ with $\phi^{-1}(H')=H_0$.
\end{proof}

In the lemma above, if $(H \leq G) = (H' \leq G')$ then the resulting $\phi$ can simply be the identity in $G$. If $(H \leq G)$ is endopure, that means that \emph{only} the identity in $G$ works. So in a way, the endopure transfers are those for which $(H \leq G)$ lives in $\overline{\la H \leq G \ra}$ only for trivial reasons, meaning that $(H \leq G)$ cannot be written as a non-trivial composition of restrictions of itself. We will invert the condition of \zcref{lemm:main-property-generate} to construct the global transfer system $T$ antigenerated by $(H \leq G)$, which will also mean that if $(H \leq G)$ is endopure, then it will \emph{only} be missing from that $T$ for trivial reasons.

\begin{constr}
\label{constr-antigen}
For a transfer $(H \lneq G)$, let $U_{H \leq G}$ denote the set of all the transfers $(H' \leq G)$ for $H \leq H' \lneq G$.

Let $R_{H \leq G}$ denote the set of transfers satisfying that $(L \leq K) \in R_{H \leq G}$ if and only if there exists a homomorphism $\phi \colon G \to K$ and a transfer $(H' \leq G) \in U_{H \leq G}$ such that $\phi^{-1}(L) = H'$. Note that $U_{H \leq G} \subset R_{H \leq G}$. Let $\underline{\la H \leq G \ra} \coloneq \mathbb{I}(\fami) \setminus R_{H \leq G}$. We refer to $\underline{\la H \leq G \ra}$ as the (global) transfer system \textit{antigenerated} by $(H \leq G)$.

Equivalently, given a transfer $(H \leq G)$, another transfer $(L \leq K)$ belongs to $\underline{\la H \leq G \ra}$ if and only if for any $\phi \colon G \to K$, the subgroup $\phi^{-1}(L) \leq G$ equals all of $G$ or does not contain $H$. That is, $(L \leq K) \notin \underline{\la H \leq G \ra}$ if and only if there exists a homomorphism $\phi \colon G \to K$ with $H \leq \phi^{-1}(L) \lneq G$. Note that this last equivalent condition is the opposite of the one in \zcref{lemm:main-property-generate}.
\end{constr}

\begin{prop}
\label{prop:anti-is-gts}
    For any transfer $(H \lneq G)$ the set $\underline{\la H \leq G \ra}$ is a global transfer system.
\end{prop}
\begin{proof}
    First we check restriction. Assume that $(L \leq K)$ belongs to $\underline{\la H \leq G \ra}$, and that we have a homomorphism $\phi \colon A \to K$. Assume for a contradiction that $(\phi^{-1}(L)\leq A) \notin \underline{\la H \leq G \ra}$. Then there exists a $\psi \colon G \to A$ such that $H \leq \psi^{-1}(\phi^{-1}(L)) \lneq G$. But $\psi^{-1}(\phi^{-1}(L))= (\phi \circ \psi)^{-1}(L)$, so this contradicts $(L \leq K) \in \underline{\la H \leq G \ra}$.

    To check that $\underline{\la H \leq G \ra}$ is closed under composition, Assume that we have $M \leq L \leq K$, with $(M \leq L), (L \leq K) \in \underline{\la H \leq G \ra}$, and consider a homomorphism $\phi\colon G \to K$. Then since $(L \leq K) \in \underline{\la H \leq G \ra}$, we know that the subgroup $\phi^{-1}(L)$ equals $G$ or does not contain $H$. If $\phi^{-1}(L)$ does not contain $H$ then neither does $\phi^{-1}(M) \leq \phi^{-1}(L)$. If instead $\phi^{-1}(L)=G$ then $\phi$ sends $G$ to $L$, and since $(M \leq L) \in \underline{\la H \leq G \ra}$, we see that $\phi^{-1}(M)$ equals $G$ or does not contain $H$. As this holds for any $\phi$, by construction $(M \leq K) \in \underline{\la H \leq G \ra}$.
\end{proof}

\begin{rem}
\label{rem:antigeneration-example}
    The global transfer system constructed above gives us a preferred choice of a global transfer system that is missing a specific transfer. We warn the reader however that this is not generally maximal among such objects, and thus calling these \emph{antigenerated} by $S$ is not entirely appropriate. An example can be seen by looking at $(1 \leq C_6)$, where $\underline{\la 1 \leq C_6 \ra}$ is trivial, so it does not contain $(C_2 \leq C_6)$ or $(C_3 \leq C_6)$, but there are global transfer systems that contain some of these but not $(1 \leq C_6)$ (the global transfer system $\overline{\la C_3 \leq C_6 \ra}$ generated by $(C_3 \leq C_6)$ would be an example). However, as we will see, all meet-irreducible global transfer systems are of this form constructed above, and in the next proposition we will show that any global transfer system $T$ that does not contain a given transfer $(H \leq G)$ must be contained in some global transfer system of this form, whose choice depends on $T$, and not just on $(H \leq G)$. This partially justifies our naming decision, as the construction gives us a very useful choice of a global transfer system that is missing a specific transfer.
\end{rem}

\begin{rem}
    In \cite[Proposition~2.12]{balchinspitz} it is shown that in the case of equivariant transfer systems the meet-irreducibles correspond to right lifts of cotransfer systems generated by a single transfer. As mentioned above, there is no analogue to the notion of cotransfer system in the global setting, but the final result of their construction coincides with what would be the analogue of our \zcref{constr-antigen}. 
\end{rem}

\begin{prop}
\label{prop-smallerumbrella}
    Let $T$ be a global transfer system, and assume that some transfer $(H \leq G)$ does not belong to $T$. Then there exists some group $G_0 \in \fami$ with $H \leq G_0 \leq G$, $H \neq G_0$, and such that $T \subset \underline{\la H \leq G_0 \ra}$.
\end{prop}
\begin{proof}
    We prove this by induction on the index of the transfer $(H \leq G)$. If the index of $(H \leq G)$ is prime, then there are no intermediate subgroups between $H$ and $G$. This means that in the notation of \zcref{constr-antigen}, $U_{H \leq G}$ contains just the transfer $(H \leq G)$. Then restriction shows that since $(H \leq G) \notin T$, we also have that any transfer in $R_{H \leq G}$ is not in $T$. Thus $T \subset \underline{\la H \leq G \ra}$.

    Assume that the statement of the proposition holds whenever the index of $(H \leq G)$ is $<N$. For a transfer $(H \leq G)$ with index $N$ and a global transfer system $T$ with $(H \leq G) \notin T$, if $U_{H \leq G}$ is contained in the complement of $T$ then we are done. Otherwise, this means that there exists some transfer $(G' \leq G)$ in $T$ with $H \leq G' \leq G$ and $G' \neq G$. This shows that $(H \leq G') \notin T$, otherwise by composition we would have $(H \leq G) \in T$. The index of $H$ in $G'$ must now be $<N$, and the induction hypothesis shows that there exists $G_0$ with $H \leq G_0 \leq G'$, $H \neq G_0$, and such that $T \subset \underline{\la H \leq G_0 \ra}$.
\end{proof}

\subsection{Joins of join-irreducibles and meets of meet-irreducibles}\label{subsec:condition}

In this section we will first show that the join-irreducible global transfer systems satisfy the condition required for \zcref{thm:FTFCA2}. Then we will do the same for the meet-irreducibles.

Recall that the join in $\mathsf{Tr}^\mathsf{Gl}(\mathcal{F})$ is given by the closure under composition of the union. Each global transfer system can be written as a join of the transfer systems generated by its individual transfers:

\begin{prop}
\label{prop-cupgens}
    Let $T$ be a global transfer system. Then
    \[T= \bigvee_{(H \leq G) \in T} \overline{\la H \leq G \ra}.\]
\end{prop}
\begin{proof}
    By construction. \qedhere
\end{proof}

From the above we can deduce that a join-irreducible must be generated by a single transfer.

\begin{prop}
\label{prop:join-irre-form}
    Any join-irreducible global transfer system $T$ is of the form $\overline{\la H \leq G \ra}$ for some transfer $(H \leq G)$.
\end{prop}
\begin{proof}
    Apply \zcref{prop-cupgens} and the definition of join-irreducible. \qedhere
\end{proof}

We just saw that each join-irreducible is generated by a single transfer. However not all transfers will generate a join-irreducible global transfer system, and most join-irreducibles can be generated by multiple different non-isomorphic transfers. Studying which transfers generate irreducibles is the goal of \zcref{subsec:irrs}, where we also deal with this multiplicity. For now we will check the condition on join-irreducibles required for \zcref{thm:FTFCA2}.

\begin{prop}
\label{prop:join-condition}
    Any global transfer system of the form $\overline{\la H \leq G \ra}$ can be written as a join of join-irreducible global transfer systems. In light of \zcref{prop-cupgens}, any global transfer system can also be written as a join of join-irreducible global transfer systems.
\end{prop}
\begin{proof}
    We prove the statement by induction on the index of the transfer $(H \leq G)$. First assume that the index of $(H \leq G)$ is prime, we will show that $\overline{\la H \leq G \ra}$ is join irreducible. If $\overline{\la H \leq G \ra} = \bigvee_{i \in I} T_i$ for some (possibly infinite) $I$ then $(H \leq G)$ is a composition of transfers in $\bigcup_{i \in I} T_i$, but since there are no intermediate subgroups between $H$ and $G$, this is a composition of length one, and $(H \leq G) \in \bigcup_{i \in I} T_i$. This means that $(H \leq G) \in T_i$ for some $i$, so $\overline{\la H \leq G \ra} =T_i$.

    Assume that the statement of the proposition holds whenever the index of $(H \leq G)$ is $<N$, and consider a transfer $(H \leq G)$ with index $N$. If $\overline{\la H \leq G \ra}$ is join-irreducible then we are done. If $\overline{\la H \leq G \ra}$ is not join-irreducible then $\overline{\la H \leq G \ra} = \bigvee_{i \in I} T_i$ for some (possibly infinite) non-trivial join of global transfer systems. This means that $H \leq G$ can be written as a (finite) composition of transfers in $\bigcup_{i \in I} T_i$. Thus we have $H= H_0 \leq H_1, \leq \dots, \leq H_M =G$, and for each $1 \leq j \leq M$ there is an $i_j \in I$ with $(H_{j-1} \leq H_j) \in T_{i_j}$. Here $M \geq 2$, otherwise $\overline{\la H \leq G \ra} = T_{i_1}$ and the join above would be trivial. We have
    \[\overline{\la H \leq G \ra} \subset \bigvee_{j=1}^M \overline{\la H_{j-1} \leq H_j \ra} \subset \bigvee_{j=1}^M T_{i_j} \subset \bigvee_{i \in I} T_i = \overline{\la H \leq G \ra},\]
    and hence $\overline{\la H \leq G \ra} = \bigvee_{j=1}^M \overline{\la H_{j-1} \leq H_j \ra}$. Since each transfer $(H_{j-1} \leq H_j)$ has index $<N$, each $\overline{\la H_{j-1} \leq H_j \ra}$ can be written as a join of join-irreducibles by the induction hypothesis, and thus so can $\overline{\la H \leq G \ra}$.
\end{proof}

We now prove analogues of the above for meet-irreducible global transfer systems, that meet-irreducibles are antigenerated by a single transfer, and that each global transfer system is a meet of meet-irreducibles.

\begin{prop}
\label{prop-capantis}
    Let $T$ be a global transfer system and let $C_T$ denote the set of all transfers $(H \leq G)$ that satisfy $U_{H \leq G} \cap T = \varnothing$. Then
    \[T= \bigcap_{(H \leq G) \in C_T} \underline{\la H \leq G \ra}.\]
\end{prop}
\begin{proof}
    If $U_{H \leq G} \cap T =\varnothing$ then also $R_{H \leq G} \cap T =\varnothing$, and so $T \subset \underline{\la H \leq G \ra}$. This shows that $T \subset \bigcap_{(H \leq G) \in C_T} \underline{\la H \leq G \ra}$.

    To show the other inclusion, assume that some transfer $(H \leq G)$ is not in $T$. Then by \zcref{prop-smallerumbrella} we have some $G_0$ with $H \leq G_0 \leq G$, $H \neq G_0$, and $T \subset \underline{\la H \leq G_0 \ra}$, so in particular $U_{H \leq G_0} \cap T =\varnothing$ and $(H \leq G_0) \in C_T$. If we let $\phi$ be the inclusion of $G_0$ into $G$, we see that $\phi^{-1}(H) = H$, and thus $(H \leq G) \in R_{H \leq G_0}$, so $(H \leq G) \notin \underline{\la H \leq G_0 \ra}$, and $(H \leq G) \notin \bigcap_{(H' \leq G') \in C_T} \underline{\la H' \leq G' \ra}$.
\end{proof}

\begin{prop}
\label{prop:meet-irre-form}
    Any meet-irreducible global transfer system $T$ is of the form $\underline{\la H \leq G \ra}$ for some transfer $(H \leq G)$.
\end{prop}
\begin{proof}
    Apply \zcref{prop-capantis} and the definition of meet-irreducible. \qedhere
\end{proof}

\begin{prop}
\label{prop:meet-condition}
    Any global transfer system of the form $\underline{\la H \leq G \ra}$ can be written as a meet of meet-irreducible global transfer systems. In light of \zcref{prop-capantis}, any global transfer system can also be written as a meet of meet-irreducible global transfer systems.
\end{prop}
\begin{proof}
    We prove the statement by induction on the index of the transfer $(H \leq G)$. Assume that the index of $(H \leq G)$ is prime, and that $\underline{\la H \leq G \ra} = \bigcap_{i \in I} T_i$ for some (possibly infinite) intersection of global transfer systems. Then $(H \leq G) \notin \underline{\la H \leq G \ra}$, so it must not belong to $T_j$ for some $j \in I$. Since $U_{H \leq G} = \{(H \leq G)\}$ we have $T_j \subset \underline{\la H \leq G \ra}$ and $T_j = \underline{\la H \leq G \ra}$. Thus $\underline{\la H \leq G \ra}$ is meet-irreducible.

    Assume that the statement of the proposition holds whenever the index of $(H \leq G)$ is $<N$, and consider a transfer $(H \leq G)$ with index $N$. If $\underline{\la H \leq G \ra}$ is irreducible then we are done. If $\underline{\la H \leq G \ra}$ is not irreducible then $\underline{\la H \leq G \ra} = \bigcap_{i \in I} T_i$ for some (possibly infinite) non-trivial intersection of global transfer systems. Each $(H' \leq G) \in U_{H \leq G}$ must not be contained in some $T_i$, we will choose one of those and denote it by $T_{i_{H'}}$. Then by \zcref{prop-smallerumbrella} there exists some group $G_{H'}$ with $H' \leq G_{H'} \leq G$, $H' \neq G_{H'}$, and such that $T_{i_{H'}} \subset\underline{\la H' \leq G_{H'} \ra}$. We have that 
    \[\underline{\la H \leq G \ra} = \bigcap_{i \in I} T_i \subset \bigcap_{(H' \leq G) \in U_{H \leq G}} T_{i_{H'}} \subset \bigcap_{(H' \leq G) \in U_{H \leq G}}\underline{\la H' \leq G_{H'} \ra}.\]
    As in the proof of \zcref{prop-capantis}, we can see that each $(H' \leq G) \in U_{H \leq G}$ is not contained in $\underline{\la H' \leq G_{H'} \ra}$ by restriction, so 
    \[\underline{\la H \leq G \ra} \; = \; \bigcap_{(H' \leq G) \in U_{H \leq G}}\underline{\la H' \leq G_{H'} \ra}.\]
    Now if $G_H=G$ then $\underline{\la H \leq G \ra} = T_{i_{H}}$ which would contradict the fact that the original intersection was non-trivial. Thus the index of $H'$ in $G_{H'}$ is $< N$ for each $(H' \leq G) \in U_{H \leq G}$, and the induction hypothesis shows that each $\underline{\la H' \leq G_{H'} \ra}$ can be written as a meet of meet-irreducibles, and therefore so can $\underline{\la H \leq G \ra}$.
\end{proof}

\subsection{Classifying irreducible transfers}\label{subsec:irrs}

Now that we have checked the conditions of \zcref{thm:FTFCA2}, we know that the structure of the lattice of global transfer systems is completely determined by the sets of join-irreducible and meet-irreducible global transfer systems, and their containment relation. Furthermore, we know that each join-irreducible (resp.\ meet-irreducible) is generated (resp.\ antigenerated) by a single transfer. So the only question left is to determine which transfers give rise to irreducibles in this way.

\begin{lemm}
\label{lemm:endopure-missing}
    Let $(H \lneq G)$ be an endopure transfer. Then $(K \leq G) \notin \overline{\la H \leq G \ra}$ for any $K$ with $H \lneq K \lneq G$.
\end{lemm}
\begin{proof}
    If $(K \leq G) \in \overline{\la H \leq G \ra}$ by \zcref{lemm:main-property-generate} we have $\phi \colon G \to G$ with $H\lneq K \leq f\phi^{-1}(H) \lneq G$, which contradicts the endopurity of $(H \leqslant G)$.
\end{proof}

\begin{prop}
\label{prop:endopure-join}
    Let $(H \lneq G)$ be an endopure transfer. Then $\overline{\la H \leq G \ra}$ is a join-irreducible global transfer system.
\end{prop}
\begin{proof}
    Assume that $\overline{\la H \leq G \ra}=\bigvee_i T_i$ for global transfer systems $T_i$. Then $(H \leq G)$ can be written as a composition of non-trivial transfers in $\cup_i T_i$. Let $(H_0 \lneq G)$ with $H \leq H_0$ be the top non-trivial transfer in this composition, then $(H_0 \leq G) \in \cup_i T_i \subset \overline{\la H \leq G \ra}$ so by \zcref{lemm:endopure-missing} we have $H_0=H$, thus $(H \leq G)$ is in one of the $T_i$, and $\overline{\la H \leq G \ra}$ equals $T_i$.
\end{proof}

\begin{defi}
    Let $H \lneq G$. We say that the transfer $(H \leq G)$ is a \emph{minimal generator} if for every $H' \leq G'$ with $\overline{\la H' \leq G' \ra} = \overline{\la H \leq G \ra}$ we have that $[G' \colon H'] \geq [G \colon H]$ and $\lvert G' \rvert \geq \lvert G \rvert$.

    In other words, among all transfers that antigenerate the same global transfer system, the transfer $(H \leq G)$ has minimal index and its top group has minimal order.
\end{defi}

\begin{prop}
\label{prop:endopure-minimal}
    Let $(H \lneq G)$ be an endopure transfer. Then $(H \leqslant G)$ is a minimal generator.
\end{prop}
\begin{proof}
    Assume that we have $H', G'$ with $\overline{\la H \leq G \ra} = \overline{\la H' \leq G' \ra}$. By \zcref{lemm:main-property-generate} there exists a homomorphism $\phi \colon G \to G'$ with $H \leq \phi^{-1}(H') \lneq G$, so that $(\phi^{-1}(H') \leq G) \in \overline{\la H' \leq G' \ra} = \overline{\la H \leq G \ra}$, and by \zcref{lemm:endopure-missing} we must have $\phi^{-1}(H')=H$. Then by \zcref{lemm:res-index} we obtain that $[G' \colon H'] \geq [G \colon H]$.

    Since $(H' \leq G') \in \overline{\la H \leq G \ra}$, we can write the former as a composition $H' =H'_0 \lneq H'_1 \lneq \dots \lneq G'$ of non-trivial transfers which are restrictions of $(H \leq G)$. Let $H'_i$ be the smallest of these intermediate groups that satisfies that $\phi^{-1}(H'_i)=G$ but $\phi^{-1}(H'_{i-1}) \lneq G$ (this is well defined because $\phi^{-1}(G')=G$ and $\phi^{-1}(H'_0)=H$). Since $(H'_{i-1} \leq H'_i)$ is a restriction of $(H \leq G)$, we have a homomorphism $\psi \colon H'_i \to G$ with $\psi^{-1}(H)=H'_{i-1}$. Then the composition $\psi \circ \phi \colon G \to G$ is well defined, and satisfies that $\phi^{-1}(\psi^{-1}(H))=\phi^{-1}(H'_{i-1})\lneq G$ and $H= \phi^{-1}(H') \leq \phi^{-1}(H'_{i-1})=\phi^{-1}(\psi^{-1}(H))$. Thus $\psi \circ \phi$ must be injective by endopurity of $H \leq G$, then $\phi$ is also injective and $\lvert G' \rvert \geq \lvert G \rvert$.
\end{proof}

\begin{prop}
\label{prop:minimal-plus-irre-endopure}
    Let $(H \lneq G)$ be such that $\overline{\la H \leq G \ra}$ is join-irreducible. Assume that for every $H' \leq G'$ with $\overline{\la H' \leq G' \ra} = \overline{\la H \leq G \ra}$ we have that $[G' \colon H'] > [G \colon H]$, or $[G' \colon H'] = [G \colon H]$ and $\lvert G' \rvert \geq \lvert G \rvert$. Then $H \leq G$ is endopure.
\end{prop}
    Here it is not enough to consider just the case where $H \lneq G$ is a minimal antigenerator, we need the slightly weaker assumption taken in the statement, otherwise we would not be able to construct a minimal generator for a join-irreducible in the next proposition.
\begin{proof}[Proof of \zcref{prop:minimal-plus-irre-endopure}]
    Assume for a contradiction that $H \leq G$ is not endopure. This means that there exists a non-injective endomorphism $\phi \colon G \to G$ with $H \leq \phi^{-1}(H) \lneq G$. Then \[\overline{\la H \leq G \ra} = \overline{\la H \leq \phi^{-1}(H) \ra} \vee \overline{\la \phi^{-1}(H) \leq G \ra}\]
    where one containment is given by composition and the other by restriction along $\phi$ and along the inclusion of $\phi^{-1}(H)$ in $G$. By minimality of the index we cannot have $\overline{\la H \leq G \ra} = \overline{\la H \leq \phi^{-1}(H) \ra}$. Thus since $\overline{\la H \leq G \ra}$ is join-irreducible we have $\overline{\la H \leq G \ra} = \overline{\la \phi^{-1}(H) \leq G \ra}$, and again by minimality of the index $\phi^{-1}(H)=H$.

    Consider now the transfer $(H \cap \Img \phi \leq \Img \phi)$. Since $\phi^{-1}(H)=H$ we have $\overline{\la H \cap \Img \phi \leq \Img \phi \ra} = \overline{\la H \leq G \ra}$, via restriction to a subgroup for one direction, and restriction along $\phi$ for the other. These two transfers must have the same index by \zcref{lemm:res-index}, so $\lvert \Img \phi \rvert \geq \lvert G \rvert$ by assumption. However this contradicts the fact that $\phi$ is not injective.\qedhere
\end{proof}

\begin{prop}
\label{prop:join-minimal}
    Let $T$ be a join-irreducible global transfer system. Then there exists an endopure transfer $(H \leqslant G)$ such that $T = \overline{\la H \leq G \ra}$.
\end{prop}
\begin{proof}
    The global transfer system $T$ must be generated by a single transfer $(H \leq G)$ by \zcref{prop:join-irre-form}. If this transfer does not have minimal index among those that generate $T$, we can replace it with one that has strictly smaller index and generates $T$, so we can assume that $(H \leq G)$ has minimal index. Then among those transfers that generate $T$ and have minimal index, we can choose one transfer with minimal order of the top group. This transfer will satisfy the conditions of \zcref{prop:minimal-plus-irre-endopure}, thus it will be endopure.
\end{proof}

\begin{prop}
\label{prop:join-unique}
    Let $H \lneq G$ be endopure, and assume that $\overline{\la H \leq G \ra} = \overline{\la H' \leq G' \ra}$ with $[G' \colon H'] = [G \colon H]$ and $\lvert G \rvert = \lvert G' \rvert$. Then there is an isomorphism $\phi \colon G \to G'$ with $\phi(H) = H'$, so the two transfers are isomorphic.
\end{prop}
\begin{proof}
    Since $(H \leq G)$ is endopure, it is a minimal generator by \zcref{prop:endopure-minimal}. Thus $H' \leq G'$ is a minimal generator of a join-irreducible global transfer system, so it is also endopure by \zcref{prop:minimal-plus-irre-endopure}.

    By \zcref{lemm:main-property-generate} there exists a homomorphism $\phi \colon G \to G'$ with $H \leq \phi^{-1}(H') \lneq G$, so that $(\phi^{-1}(H') \leq G) \in \overline{\la H' \leq G' \ra} = \overline{\la H \leq G \ra}$, and by \zcref{lemm:endopure-missing} we must have $\phi^{-1}(H')=H$. By the same argument we obtain $\psi \colon G' \to G$ satisfying that $\psi^{-1}(H)=H'$. Then $\psi \circ \phi$ is injective by endopurity of $H \leq G$, thus $\phi$ is injective, and an isomorphism with $\phi(H) = H'$ since $\lvert G \rvert = \lvert G' \rvert$.
\end{proof}

Putting all of the above together yields that any join-irreducible global transfer system is generated by a single endopure minimal generator.

\begin{prop}
\label{prop:join-bijection}
    The assignment that sends an endopure transfer $H \leq G$ to $\overline{\la H \leq G \ra}$ is a bijection between the set of isomorphism classes of endopure transfers and the set of join-irreducible global transfer systems.
\end{prop}
\begin{proof}
    If $H \leq G$ is endopure then $\overline{\la H \leq G \ra}$ is join-irreducible by \zcref{prop:endopure-join}. This assignment is surjective by \zcref{prop:join-minimal}. It is injective by \zcref{prop:endopure-minimal} and \zcref{prop:join-unique}.    
\end{proof}

We can also prove the dual statements of all of the above for meet-irreducible elements. Note that although most of the statements in the rest of this subsection are analogous to those that we just proved for join-irreducibles, some require different arguments. To that end, we will need a few technical results related to endopure transfers and antigenerated global transfer systems, which we check now. 

\begin{lemm}
\label{lemm:proper-umbrella}
    Let $(H \lneq G)$ be endopure. Then $(H \leq G_0) \in \underline{\la H \leq G \ra}$ for each $G_0$ with $H \leq G_0 \lneq G$.
\end{lemm}
\begin{proof}
    Assume that $(H \leq G_0) \notin \underline{\la H \leq G \ra}$, then by construction we have a homomorphism $\phi \colon G \to G_0$ with $H \leq \phi^{-1}(H) \lneq G$, and then composing $\phi$ with the subgroup inclusion $G_0\lneq G$ yield a non-injective homomorphism that contradicts the fact that $H \leq G$ is endopure.
\end{proof}

\begin{lemm}
\label{lemm:antigenerator-missing}
    Let $(H \lneq G)$ be endopure, and let $T$ be a global transfer system. If $\underline{\la H \leq G \ra} \subset T$ and $(H \leq G) \notin T$ then $T = \underline{\la H \leq G \ra}$
\end{lemm}
\begin{proof}
    If $(H \leq G) \notin T$ then by \zcref{prop-smallerumbrella} there is a subgroup $G_0$ with $H \lneq G_0 \leq G$ and $T \subset \underline{\la H \leq G_0 \ra}$. By \zcref{lemm:proper-umbrella} if $G_0 \neq G$ then $(H \leq G_0) \in \underline{\la H \leq G \ra} \subset T \subset \underline{\la H \leq G_0 \ra}$ which is a contradiction, thus $G_0=G$  and $\underline{\la H \leq G \ra} \subset T \subset \underline{\la H \leq G \ra}$.
\end{proof}

\begin{prop}
\label{prop:endopure-meet}
    Let $(H \lneq G)$ be an endopure transfer. Then $\underline{\la H \leq G \ra}$ is a meet-irreducible global transfer system.
\end{prop}
\begin{proof}
    If $\underline{\la H \leq G \ra} = \bigcap_{i \in I} T_i$, then the transfer $(H \leq G)$ must be missing from at least one of the global transfer systems $T_i$, which must then be equal to $\underline{\la H \leq G \ra}$ by \zcref{lemm:antigenerator-missing}.
\end{proof}

It is immediate that two transfers that are mutual restrictions generate the same global transfer system. It is however not so straightforward to see that they antigenerate the same global transfer system. This is because it is not universally true that if $(\phi^{-1}(L) \leq G)$ is a restriction of $(L \leq K)$, then $\underline{\la \phi^{-1}(L) \leq G \ra} \subset \underline{\la L \leq K \ra}$. This will hold if both transfers have the same index though.

\begin{lemm}
    \label{lemm:contain-antigenerated}
    Let $\phi \colon G \to K$ and let $L \leq K$. If $[K \colon L] = [G \colon \phi^{-1}(L)]$ then $\underline{\la \phi^{-1}(L) \leq G \ra} \subset \underline{\la L \leq K \ra}$.
\end{lemm}
Note that the condition on the index is essential, take $K=C_6$, $L=1$, and $\phi \colon C_2 \to C_6$ the subgroup inclusion to see this.
\begin{proof}[Proof of \zcref{lemm:contain-antigenerated}]
    Let $(K_0 \lneq K) \in U_{L \leq K}$. By \zcref{lemm:res-index} we have $[G \colon \phi^{-1}(K_0)] \leq [K \colon K_0]$ and $[\phi^{-1}(K_0)\colon \phi^{-1}(L)] \leq [K_0 \colon L]$, but \[[G \colon \phi^{-1}(K_0)] [\phi^{-1}(K_0)\colon \phi^{-1}(L)] = [G \colon \phi^{-1}(L)] =[K \colon L] = [K \colon K_0] [K_0 \colon L]\] so $[G \colon \phi^{-1}(K_0)] = [K \colon K_0] > 1$. Thus by construction $(K_0 \lneq K) \in R_{\phi^{-1}(L) \leq G}$, and we obtain that $U_{L \leq K} \subset R_{\phi^{-1}(L) \leq G}$ so $\underline{\la \phi^{-1}(L) \leq G \ra} \subset \underline{\la L \leq K \ra}$.
\end{proof}

\begin{lemm}
    \label{lemm:mutual-restriction}
    Let $(H \leq G)$ and $(H' \leq G')$ be two transfers, such that each is a restriction of the other. Then they have the same index and antigenerate the same global transfer system.
\end{lemm}
\begin{proof}
    By \zcref{lemm:res-index} these two transfers have the same index. Thus by \zcref{lemm:contain-antigenerated} the global transfer systems that they generate are mutually contained in each other.
\end{proof}

\begin{defi}
    Let $H \lneq G$. We say that the transfer $(H \leq G)$ is a \emph{minimal antigenerator} if for every $H'\leq G'$ with $\underline{\la H' \leq G' \ra} = \underline{\la H \leq G \ra}$ we have that $[G' \colon H'] \geq [G \colon H]$ and $\lvert G' \rvert \geq \lvert G \rvert$.

    In other words, among all transfers that antigenerate the same global transfer system, the transfer $(H \leq G)$ has minimal index and its top group has minimal order. 
\end{defi}

\begin{prop}
\label{prop:endopure-minimalanti}
    Let $(H \lneq G)$ be an endopure transfer. Then $(H \leqslant G)$ is a minimal antigenerator.
\end{prop}
\begin{proof}
    Consider a transfer $(H' \leq G')$ with $\underline{\la H' \leq G' \ra} = \underline{\la H \leq G \ra}$. Assume without loss of generality that $H'$ is maximal among subgroups of $G'$ that satisfy that $\underline{\la H' \leq G' \ra} = \underline{\la H \leq G \ra}$ (otherwise we can replace $H'$). 
    
    Since $\underline{\la H' \leq G' \ra} = \underline{\la H \leq G \ra}$ there exists a homomorphism $\phi_1 \colon G \to G'$ with $H \leq \phi_1^{-1}(H') \lneq G$ and another homomorphism $\phi_2 \colon G' \to G$ with $H' \leq \phi_2^{-1}(H) \lneq G'$. We have $\underline{\la H \leq G \ra} = \underline{\la H' \leq G' \ra} \subset \underline{\la \phi_2^{-1}(H) \leq G' \ra}$ by construction, and $(H \leq G) \notin \underline{\la \phi_2^{-1}(H) \leq G' \ra}$ by restriction along $\phi_2$, so by \zcref{lemm:antigenerator-missing} we obtain that $\underline{\la H \leq G \ra} = \underline{\la \phi_2^{-1}(H) \leq G' \ra}$. By maximality of $H'$ we obtain that $H' = \phi_2^{-1}(H)$, and thus $H \leq \phi_1^{-1}(\phi_2^{-1}(H)) \lneq G$. Since $(H \leq G)$ is endopure, we deduce that $\phi_2 \circ \phi_1$ must be injective, and thus an isomorphism. Therefore $H= \phi_1^{-1}(\phi_2^{-1}(H))= \phi_1^{-1}(H')$ and by \zcref{lemm:res-index} we have $[G' \colon H'] \geq [G \colon H]$. Additionally $\phi_1$ must be injective and thus $\lvert G' \rvert \geq \lvert G \rvert$.

\end{proof}

\begin{prop}
\label{prop:minimal-plus-meet-irre-endopure}
    Let $(H \lneq G)$ be such that $\underline{\la H \leq G \ra}$ is meet-irreducible. Assume that for every $H' \leq G'$ with $\underline{\la H' \leq G' \ra} = \underline{\la H \leq G \ra}$ we have $[G' \colon H'] > [G \colon H]$, or $[G' \colon H'] = [G \colon H]$ and $\lvert G' \rvert \geq \lvert G \rvert$. Then $H \leq G$ is endopure.
\end{prop}
\begin{proof}
    Assume for a contradiction that $H \leq G$ is not endopure. This means that there exists a non-injective endomorphism $\phi \colon G \to G$ with $H \leq \phi^{-1}(H) \lneq G$. Then \[\underline{\la H \leq G \ra} = \underline{\la \phi^{-1}(H) \leq G\ra} \cap \underline{\la \la H \cup \Img \phi \ra \leq G \ra}\] where one containment is by construction. For any $(H' \leq G) \in U_{H \leq G}$ if $\Img \phi \leq H'$ then $(H' \leq G) \notin \underline{\la \la H \cup \Img \phi \ra \leq G \ra}$. Otherwise $\phi^{-1}(H) \leq \phi^{-1}(H') \lneq G$ and $(H' \leq G) \notin \underline{\la \phi^{-1}(H) \leq G\ra}$, which shows the other containment.
    
    Since $\phi^{-1}(H) \lneq G$, $\Img \phi$ is not contained in $H$, and the transfer $(\la H \cup \Img \phi \ra \leq G)$ has a strictly smaller index than $(H \leq G)$. Thus by assumption we cannot have $\underline{\la H \leq G \ra} = \underline{\la \la H \cup \Img \phi \ra \leq G \ra}$. Since $\underline{\la H \leq G \ra}$ is meet-irreducible we must have $\underline{\la H \leq G \ra} = \underline{\la \phi^{-1}(H) \leq G \ra}$, and once again by assumption $\phi^{-1}(H)=H$.

    Consider now the transfer $(H \cap \Img \phi \leq \Img \phi)$. This transfer restricts to $(H \leq G)$ through $\phi$ because $\phi^{-1}(H)=H$, and we can restrict $(H \leq G)$ through the subgroup inclusion of $\Img \phi$ into $G$ to obtain $(H \cap \Img \phi \leq \Img \phi)$. Thus by \zcref{lemm:mutual-restriction} these transfers have the same index and $\underline{\la H \cap \Img \phi \leq \Img \phi \ra} = \underline{\la H \leq G \ra}$. But $\lvert \Img \phi \rvert < \lvert G \rvert$ because $\phi$ is not injective, thus this contradicts the fact that $H \lneq G$ is a minimal generator.
\end{proof}

\begin{prop}
\label{prop:meet-minimal}
    Let $T$ be a meet-irreducible global transfer system. Then there exists an endopure transfer $(H \leqslant G)$ such that $T = \underline{\la H \leq G \ra}$.
\end{prop}
\begin{proof}
    This proof is identical to that of \zcref{prop:join-minimal}, now using \zcref{prop:meet-irre-form} and \zcref{prop:minimal-plus-meet-irre-endopure} instead.\qedhere
\end{proof}

\begin{prop}
\label{prop:meet-unique}
    Let $(H \lneq G)$ be endopure, and assume that $\underline{\la H \leq G \ra} = \underline{\la H' \leq G' \ra}$ with $[G' \colon H'] = [G \colon H]$ and $\lvert G \rvert = \lvert G' \rvert$. Then there is an isomorphism $\psi \colon G' \to G$ with $\psi(H') = H$, so the two transfers are isomorphic.
\end{prop}
\begin{proof}
    Since $(H \leq G)$ is endopure, it is a minimal antigenerator by \zcref{prop:endopure-minimalanti}. Thus $H' \leq G'$ is a minimal antigenerator of a meet-irreducible global transfer system, so it is also endopure by \zcref{prop:minimal-plus-meet-irre-endopure}.

    Since $\underline{\la H' \leq G' \ra} = \underline{\la H \leq G \ra}$ there exist homomorphisms $\phi \colon G \to G'$ with $H \leq \phi^{-1}(H') \lneq G$ and $\psi \colon G' \to G$ with $H' \leq \psi^{-1}(H) \lneq G'$. We have $\underline{\la H \leq G \ra} = \underline{\la H' \leq G' \ra} \subset \underline{\la \psi^{-1}(H) \lneq G' \ra}$ and $(H \leq G) \notin \underline{\la \psi^{-1}(H) \lneq G' \ra}$ by restriction along $\psi$, so by \zcref{lemm:antigenerator-missing} we obtain that $\underline{\la H \leq G \ra} = \underline{\la \psi^{-1}(H) \lneq G' \ra}$. Since $(H \leq G)$ is a minimal antigenerator and $[G' \colon H'] = [G \colon H]$ we must have $\psi^{-1}(H)=H'$. By the same argument $\phi^{-1}(H')=H$. Then $H'=\psi^{-1}(\phi^{-1}(H')) \lneq G'$, and the endopurity of $H' \leq G'$ shows that the composition $\phi \circ \psi$ must be injective. Thus $\psi$ is injective, and an isomorphism, with $\psi(H') = H$.
\end{proof}

Now putting all of the above together yields that any meet-irreducible global transfer system is also generated by a single endopure minimal generator.

\begin{prop}
\label{prop:meet-bijection}
    The assignment that sends an endopure transfer $(H \leq G)$ to $\underline{\la H \leq G \ra}$ is a bijection between the set of isomorphism classes of endopure transfers and the set of meet-irreducible global transfer systems.
\end{prop}
\begin{proof}
    If $H \leq G$ is endopure then $\underline{\la H \leq G \ra}$ is meet-irreducible by \zcref{prop:endopure-meet}. This assignment is surjective by \zcref{prop:meet-minimal}. It is injective by \zcref{prop:endopure-minimalanti} and \zcref{prop:meet-unique}.    
\end{proof}

All that is left to do is to provide a condition for when join-irreducible elements are contained in meet-irreducible elements, at which point we have the entire description of the formal context for global transfer systems.

\begin{prop}
\label{prop:containment-condition}
    Let $(H \leq G)$ and $(L \leq K)$ be two elements of $\mathsf{EP}(\mathcal{F})$. We have that $\overline{\la L \leq K \ra} \subset \underline{\la H \leq G \ra}$ if and only if there are no homomorphisms $\phi \colon G \to K$ with $H \leq \phi^{-1}(L) \lneq G$.

    Equivalently, $\overline{\la L \leq K \ra} \subset \underline{\la H \leq G \ra}$ if and only if for any homomorphism $\phi \colon G \to K$, if $\phi(H) \leq L$, then $\phi(G) \leq L$.
\end{prop}
\begin{proof}
    We have that $\overline{\la L \leq K \ra} \subset \underline{\la H \leq G \ra}$ if and only if $(L \leq K) \in \underline{\la H \leq G \ra}$ if and only if there are no homomorphisms $\phi \colon G \to K$ with $H \leq \phi^{-1}(L) \lneq G$ by construction.
\end{proof}

We summarise the results of this section in the following theorem:

\begin{thm}\label{thm:summary}
    Let $\mathcal{F}$ be a family. Then there is a canonical isomorphism \[\varepsilon_\fami \colon \mathsf{Tr}^\mathsf{Gl}(\fami) \to \conlat(\mathsf{EP}(\fami), \mathsf{EP}(\fami), {\leq_{\mathsf{EP}}})\] given by
    \[T \mapsto (\{(L\leq K) \in \mathsf{EP}(\fami) \mid (L \leq K) \in T\},\{(H\leq G) \in \mathsf{EP}(\fami) \mid  (H' \leq G) \notin T \; \text{for every} \; H \leq H' \lneq G \}).\]
        
    The relation $\leq_{\mathsf{EP}}$ in $\mathsf{EP}(\fami) \times \mathsf{EP}(\fami)$ is given by $(L \leq K) \leq_{\mathsf{EP}} (H \leq G)$ if there are no homomorphisms $\phi \colon G \to K$ with $H \leq \phi^{-1}(L) \lneq G$.
\end{thm}
\begin{proof}
    We can apply \zcref{thm:FTFCA2} by \zcref{prop:join-condition} and \zcref{prop:meet-condition}. The constructions $\overline{\la - \ra}$ and $\underline{\la - \ra}$ induce bijections from $\mathsf{EP}(\fami)$ to $J(\mathsf{Tr}^\mathsf{Gl}(\fami))$ and $M(\mathsf{Tr}^\mathsf{Gl}(\fami))$ by \zcref{prop:join-bijection} and \zcref{prop:meet-bijection} respectively. The induced relation between irreducibles corresponds to $\leq_{\mathsf{EP}}$ by \zcref{prop:containment-condition}.
\end{proof}

\section{Examples}
\label{sec:examples}

We can now use our characterization of the lattice of global transfer systems in terms of endopure transfers to study some interesting families, obtaining novel computations and some surprising examples of behaviour.

\subsection{Change of family and filtrations}\label{subsec:filts}

Restricting global transfer systems to a subfamily $\mathcal{E} \subset \fami$ preserves containment, inducing a morphism of posets which was denoted by $U^{\fami}_{\mathcal{E}}$ in \cite[Lemma~4.2]{abeliantransfers}. In \cite[Proposition~4.4]{abeliantransfers} a right adjoint (as posets) of this restriction morphism was constructed, which was in fact an embedding of $\mathsf{Tr}^\mathsf{Gl}(\mathcal{E})$ in $\mathsf{Tr}^\mathsf{Gl}(\fami)$. Here we will consider a different embedding, given by the left adjoint.

\begin{lemm}
    The restriction homomorphisms $U^{\mathcal{F}}_{\mathcal{E}}$ have left adjoints $L^{\mathcal{F}}_{\mathcal{E}}$. These left adjoints are embeddings with $U^{\mathcal{F}}_{\mathcal{E}} \circ L^{\mathcal{F}}_{\mathcal{E}} = \id_\mathcal{E}$.
\end{lemm}
\begin{proof}
    Let $L^{\mathcal{F}}_{\mathcal{E}}(T) = \overline{\la T \ra}$ be the global transfer system generated by $T$ in $\mathcal{F}$. We can see that $U^{\mathcal{F}}_{\mathcal{E}} \circ L^{\mathcal{F}}_{\mathcal{E}} = \id_\mathcal{E}$, because if $(H \leq G) \in U^{\mathcal{F}}_{\mathcal{E}}(L^{\mathcal{F}}_{\mathcal{E}}(T))$ then it is a composition of restrictions of transfers in $T$, but since $\mathcal{E}$ is closed under subgroups, those restrictions were already in $T$. We also see that $(L^{\mathcal{F}}_{\mathcal{E}} \circ U^{\mathcal{F}}_{\mathcal{E}})(T) \subset T$ by construction of $L^{\mathcal{F}}_{\mathcal{E}}$, so $L^{\mathcal{F}}_{\mathcal{E}}$ is left adjoint to $U^{\mathcal{F}}_{\mathcal{E}}$.
\end{proof}

Note that since $U^{\fami}_{\mathcal{E}}$ has both adjoints, it preserves both joins and meets of global transfer systems, making it a lattice homomorphism. A lattice homomorphism need not preserve meet- or join-irreducibles, and indeed these are not normally preserved by $U^{\fami}_{\mathcal{E}}$. However our classification of irreducibles in terms of endopure transfers makes it straightforward to see that if $\fami$ has no new endopure transfers, then $U^{\fami}_{\mathcal{E}}$ induces bijections between irreducibles, and is thus an isomorphism of lattices.

\begin{coro}
    Let $\mathcal{E} \subset \fami$ be a subfamily. If every endopure transfer in $\fami$ is contained in $\mathcal{E}$ then the homomorphism of lattices $U^{\fami}_{\mathcal{E}} \colon \mathsf{Tr}^\mathsf{Gl}(\fami) \to \mathsf{Tr}^\mathsf{Gl}(\mathcal{E})$ induced by restriction of global transfer systems is an isomorphism.
\end{coro}
\begin{proof}
    The inverse of the isomorphism $\varepsilon_\fami$ from \zcref{thm:summary} sends an element $(A, B)$ of the lattice $\conlat(\mathsf{EP}(\fami), \mathsf{EP}(\fami), {\leq_{\mathsf{EP}}})$ to the global transfer system $\bigvee_{(H \leq G) \in A} \overline{\la H \leq G \ra}$. Thus the composition $(\varepsilon_{\mathcal{E}})^{-1} \circ \varepsilon_\fami$ sends an $\fami$- global transfer system $T$ to $\bigvee_{(H \leq G) \in \mathsf{EP}(\mathcal{E}) \cap T} \overline{\la H \leq G \ra}$ in $\mathsf{Tr}^\mathsf{Gl}(\mathcal{E})$. Restriction preserves joins and $U^{\fami}_{\mathcal{E}}(\overline{\la H \leq G \ra}) = \overline{\la H \leq G \ra}$ when $(H \leq G)$ is in $\mathcal{E}$, so $U^{\fami}_{\mathcal{E}}(T)=\bigvee_{(H \leq G) \in \mathsf{EP}(\mathcal{E}) \cap T} \overline{\la H \leq G \ra}$ and $U^{\fami}_{\mathcal{E}} = (\varepsilon_{\mathcal{E}})^{-1} \circ \varepsilon_\fami$.
\end{proof}

We can filter our family, by group order for example, to obtain an easier description of the lattice of global transfer systems, and we can study the limit behaviour of these filtrations.

\begin{prop}
    Let $\fami$ be a family of finite groups, and let $\{\fami_n\}_{n \in \mathbb{N}}$ be a nested sequence of subfamilies with $\cup_n \fami_n = \fami$. The lattice $\mathsf{Tr}^\mathsf{Gl}(\fami)$ is isomorphic to the limit of the lattices $\mathsf{Tr}^\mathsf{Gl}(\fami_n)$ along the restriction functors.
\end{prop}
\begin{proof}
    The limit of the $\mathsf{Tr}^\mathsf{Gl}(\fami_n)$ along the restriction functors is the limit of the underlying sets, with meets and joins taken pointwise. The set function $\mathsf{Tr}^\mathsf{Gl}(\fami) \to \lim_n \mathsf{Tr}^\mathsf{Gl}(\fami_n)$ induced by the restrictions is injective because $\cup_n \fami_n = \fami$. Taking the global transfer system $\overline{\la \cup_n T_n \ra}$ shows that it is surjective, and it preserves meets and joins because the restriction homomorphisms do.
\end{proof}

\begin{defi}
    Given a family $\fami$ with finitely many endopure transfers, let $\delta(\fami)$ denote the density of the binary matrix $C(\fami)$ representing the formal context $(\mathsf{EP}(\fami), \mathsf{EP}(\fami), {\leq_{\mathsf{EP}}})$. The entry of $C(\fami)$ corresponding to $((L \leq K), (H \leq G))$ in $\mathsf{EP}(\fami) \times \mathsf{EP}(\fami)$ is $1$ if $(L \leq K) \leq_{\mathsf{EP}} (H \leq G)$, and $0$ otherwise. The density $\delta(\fami)$ is the number of ones in $C(\fami)$ divided by $\lvert \mathsf{EP}(\fami) \rvert^2$.
\end{defi}

\begin{ex}
    As $\overline{\la (H \leq G) \ra}$ is never contained in $\underline{\la (H \leq G) \ra}$, the diagonal of $C(\fami)$ consists of zeroes, and thus $\delta(\fami)$ is never $1$. If $\delta(\fami)=0$, then $\mathsf{Tr}^\mathsf{Gl}(\fami)$ is the complete lattice with two elements.
\end{ex}

Even though $\delta(\fami)$ will not be $0$ or $1$ in any interesting case, we can study the limit of the density of increasing families. This usually is a good indicator of the ``complexity'' of the resulting lattices, and can be used to form non-trivial bounds on the number of concepts in the formal context. We refer the reader to \cite{balchinspitz} for more details on this.

\begin{ex}
\label{ex:density-cyclic}
    For an intermediate example, consider $\fami=\mathcal{C}yc_p$, the family of cyclic $p$-groups. By \zcref{ex:cpn} the endopure transfers here are precisely $(C_{p^{n-1}} \leq C_{p^n})$. Let $\fami_n=\Sub(C_{p^n})$. The matrix $C(\fami_n)$ can be arranged to be strictly lower triangular, and its density is $\frac{n(n-1)}{2n^2}$, thus $\lim_n \delta(\fami_n) = \frac{1}{2}$.
\end{ex}

\subsection{Subgroups of a fixed group}

In this section we will consider the simplest family, namely $\Sub(G)$ for a fixed group $G$. First, we complete \zcref{ex:epofsl2} and compute the number of global transfer systems for the family $\Sub(\mathrm{SL}_2(\mathbb{F}_5))$, for which we know there are 21 join-irreducibles and meet-irreducibles. The following $21\times21$ matrix is the formal context matrix for this family, where the ordering of the rows and columns correspond exactly to the ordering in the order of \zcref{ex:epofsl2}:
\setcounter{MaxMatrixCols}{21}
\[ \resizebox{0.4\textwidth}{!}{$\begin{bmatrix} 0&0&0&0&0&0&0&0&0&0&0&1&1&1&1&1&1&1&1&1&1\\ 1&0&1&0&1&0&0&0&0&0&0&1&1&1&1&1&1&1&1&1&1\\ 1&1&0&1&1&0&1&1&0&1&1&1&1&1&1&1&1&1&1&1&1\\ 1&1&1&0&1&1&0&1&0&0&0&1&1&1&1&1&1&1&1&1&1\\ 1&1&1&1&0&1&1&0&1&0&0&1&1&1&1&1&1&1&1&1&1\\ 1&1&1&1&1&0&1&1&0&1&1&1&1&1&1&1&1&1&1&1&1\\ 1&1&1&1&1&1&0&1&1&0&1&1&1&1&1&1&1&1&1&1&1\\ 1&1&1&1&1&1&1&0&1&0&0&1&1&1&1&1&1&1&1&1&1\\ 1&1&1&1&1&1&1&1&0&1&1&1&1&1&1&1&1&1&1&1&1\\ 1&1&1&1&1&1&1&1&1&0&1&1&1&1&1&1&1&1&1&1&1\\ 1&1&1&1&1&1&1&1&1&1&0&1&1&1&1&1&1&1&1&1&1\\ 0&1&0&1&0&1&1&1&1&1&1&0&1&1&1&1&1&1&1&0&1\\ 0&0&0&0&0&0&0&0&0&0&0&1&0&1&1&1&1&0&1&0&0\\ 0&0&0&0&0&0&0&0&0&0&0&0&1&0&1&0&1&0&0&0&0\\ 0&0&0&0&0&0&0&0&0&0&0&1&1&1&0&1&1&1&0&1&1\\ 0&0&0&0&0&0&0&0&0&0&0&0&1&1&1&0&0&1&1&0&0\\ 0&1&0&0&0&1&1&1&0&1&0&0&1&1&1&1&0&1&1&0&1\\ 1&1&1&0&1&1&0&1&0&0&0&1&1&1&1&1&1&0&1&1&1\\ 1&1&1&0&1&1&0&1&0&0&0&1&1&1&1&1&1&1&0&1&1\\ 1&1&1&1&0&1&1&0&1&0&0&1&1&1&1&1&1&1&1&0&0\\ 1&1&1&1&1&1&1&0&1&0&0&1&1&1&1&1&1&1&1&1&0 \end{bmatrix}$} \]
On the first row we record when $\overline{\la K \leq H \ra} \subset \underline{\la  1 \leq \SL_2(\mathbb{F}_5)  \ra}$. When $H = \SL_2(\mathbb{F}_5)$ we do not get this inclusion, hence the 0s, while as $H$ is a proper subgroup of $\SL_2(\mathbb{F}_5)$ we do get this inclusion, hence the 1s. The rest of the matrix is filled out similarly\footnote{The authors used \texttt{GAP} to compute this matrix, as well as the other matrices in this section.}.

After inputting this matrix into any Formal Concept Analysis software we conclude that $|\mathsf{Tr}^\mathsf{Gl}(\mathrm{SL}_2(\mathbb{F}_5)| = 326$. This should be compared to the number of equivariant transfer systems for $\mathrm{SL}_2(\mathbb{F}_5)$ which is 29014. This is a general phenomenon, it is very rare for an equivariant transfer system to satisfy the condition to be a global transfer system.

In \zcref{tab:fam} we provide data for some families of finite groups along with data regarding the number of equivariant transfer systems for the group. Of course, for alternating and symmetric groups, one does not expect some formula for any of the terms appearing (indeed, we do not even explicitly know the subgroups of $S_n$ in general). For the generalized quaternion groups, we at least recognise patterns within $|\mathsf{EP}(G)|$ and $|\Sub(\mathsf{EP}(G))|$, which can be read off easily from the subgroup descriptions, along with the observation that their endocore is exactly $C_2$.

\begin{table}[h]
\begin{tabular}{|c|c|c|c|c|}
\hline
Group $G$& $|\mathsf{EP}(G)|$ & $|\mathsf{EP}(\mathrm{Sub}(G))|$ & $|\mathsf{Tr}^\mathsf{Gl}(G)|$ & $|\mathsf{Tr}(G)|$\\ \thickhline
$C_p$ & 1 & 1 & 2 &2\\ \hline
$C_{p^2}$ & 1 & 2 & 3 &5\\ \hline
$C_{p^n}$ & 1 & $n$ & $n+1$ & $\frac{1}{n+2}\binom{2n+2}{n+1}$ \cite{BBR}\\ \hline
$C_p^2$ & 0 & 1 & 2 & $2^{p+2}+p+1$ \cite{bao}\\ \thickhline
$A_4$ & 1 & 3 & 5 & 20\\ \hline
$A_5$ & 8 & 14 & 73 & 987\\ \hline
$A_6$ & 15 & 21 & 2891 & 37799146070\\ \thickhline
$S_3$ & 1 & 3 & 5 & 9\\ \hline
$S_4$ & 2 & 7 & 12 & 8691\\ \hline
$S_5$ & 9 & 27 & 1346 & 83598202\\ \hline
$S_6$ & 20 & 64 & 2545285 & unknown\\ \thickhline
$Q_8$ & 2 & 4 & 5 & 68\\ \hline
$Q_{16}$ & 5 & 10 & 15 & 1090\\ \hline
$Q_{32}$ & 7 & 18 & 60 & 23122\\ \hline
$Q_{64}$ & 9 & 28 & 272 &621752 \\ \hline
$Q_{128}$ & 11 & 40 & 1287 & unknown \\ \hline
\end{tabular}\caption{Data for global transfer systems for the family $\Sub(G)$.}\label{tab:fam}
\end{table}

\begin{rem}
    It is, of course, natural to ask what properties the lattice of global transfer systems has. Such questions are addressed in \cite{MR5020213} for equivariant transfer systems where it is shown that $\mathsf{Tr}(G)$ is distributive if and only if $G = C_p$ (see \cite[Proposition 4.1]{MR5020213}). In the global case, we see that $\mathsf{Tr}^\mathsf{Gl}(\fami)$ is totally ordered, and thus distributive, in some of the examples computed in this paper, but this is not always the case. Indeed, we see the failure of distributivity for $\fami = \Sub(\SL_2(\mathbb{F}_5))$. In particular
    \[ \resizebox{1\textwidth}{!}{$\left(\overline{\la 1 \leq C_2 \ra} \vee \overline{\la C_2 \leq \mathrm{SL}_2(\mathbb{F}_5) \ra}\right) \wedge \overline{\la 1 \leq \mathrm{SL}_2(\mathbb{F}_5) \ra} \neq \left(\overline{\la 1 \leq C_2 \ra} \wedge \overline{\la 1 \leq \mathrm{SL}_2(\mathbb{F}_5) \ra}\right) \vee \left(\overline{\la C_2 \leq \mathrm{SL}_2(\mathbb{F}_5) \ra} \wedge \overline{\la 1 \leq \mathrm{SL}_2(\mathbb{F}_5) \ra}\right)$}\]
    as the first term contains $(1 \leq \mathrm{SL}_2(\mathbb{F}_5))$ but the second doesn't.
\end{rem}

\subsection{Finite abelian groups of bounded order}\label{subsec:ab}

Let $\mathcal{A}(n)$ denote the family of abelian groups of order less than or equal to $n$, and $\mathcal{A}$ the family of all finite abelian groups. In \cite{abeliantransfers} the second author calculates that $\mathsf{Tr}^\mathsf{Gl}(\mathcal{A}) \cong \prod_p (\omega + 1)$. Using the theory built in this paper one can conclude that $\mathsf{Tr}^\mathsf{Gl}(\mathcal{A}_n) \cong \prod_{\substack{p \leq n \\ p \,\mathrm{ prime}}} [f_p(n)] \quad \mathrm{where} \quad f_n(p) \coloneq \left\lfloor \frac{\log n}{\log p} \right\rfloor$, and indeed the limit of these lattices as $n \to \infty$ retrieves the calculation of \cite{abeliantransfers} using \zcref{subsec:filts}. The goal of this section is to prove results regarding the density of the matrix $C(\mathcal{A}(n))$.

\begin{thm}\leavevmode
    \[
        \delta(\mathcal{A}(n)) \sim 1-\frac{\log n}{n}
    \]
    In particular, as $n\to\infty$, we have that $\delta(\mathcal{A}(n)) \to 1$.
\end{thm}

\begin{proof}
    We will require two calculations, first we will need to know the behaviour of the number of endopure transfers in $\mathcal{A}(n)$, and secondly we will count how many ones occur in the context matrix.

    We will begin by tackling $|\mathsf{EP}(\mathcal{A}(n))|$. This is counted by the sequence
    \[
    a(n) \coloneq \# \{\text{prime powers} \leq n\}.
    \]
    This is clear as the group $C_{p^k}$ contributes the endopure transfer $(C_{p^{k-1}} \leq C_{p^k})$, and these are the only endopure transfers that appear by \zcref{prop:abendo}. One sees that $a(n) \sim \frac{n}{\log n}$ using the prime number theorem.

    We next compute the number of zeros in the context matrix, which we denote $Z(n)$. This is akin to determining when $\overline{\la C_{p^{k-1}} \leq C_{p^k} \ra}$ is not contained in $\underline{\la C_{q^{\ell-1}} \leq C_{q^\ell}\ra }$. Given \zcref{prop:containment-condition} this is the case whenever $p=q$ and $\ell \leq k$. Thus
    \[
    Z(n+1) = Z(n) + b(n)
    \]
    where
    \[
    b(n) \coloneq \begin{cases} m & n=p^m \\ 0 & \text{else}\end{cases}
    \]

It remains to determine the asymptotic growth of $Z(n)$. Writing $Z(n) = \sum_{x \leq n} b(x)$. As non-prime powers contribute 0, and $p^m$ contributes $m$, $Z(n) = \sum_{p^m \leq n} m$. Regrouping the sum to be dependent on the exponent, that is, we obtain $Z(n) = \sum_{m \geq 1} \sum_{p^m \leq n} m $. For a fixed $m$, every prime $p$ satisfying $p^m \leq n$ contributes $m$. But the condition that $p^m \leq n$ is equivalent to the condition that $p \leq n^{1/m}$. The number of primes counted by this is $\pi(n^{1/m})$ Therefore for fixed $m$ we see a contribution of $m \pi(n^{1/m})$, summing over all exponents we get $Z(n) = \sum_{m\geq 1} m \pi(n^{1/m})$. Now 
\[
Z(n) = \pi(n) + 2 \pi(n^{1/2}) + 3 \pi(n^{1/3}) + \cdots
\]
and again using the prime number theorem, we know the first term is asymptotically $\frac{n}{\log n}$ and this in fact dominates as $m\pi(n^{1/m}) \ll \frac{m^2 n^{1/m}}{\log n}$. Hence we conclude that $Z(n) \sim \frac{n}{\log n}$ as required.

We are finally in a position to compute the density as
\[
\delta(n) \sim 1-\frac{Z(n)}{a(n)^2} \sim1-\dfrac{\frac{n}{\log n}}{(\frac{n}{\log n})^2} =1- \frac{\log n}{n}
\]
as claimed.
\end{proof}

\subsection{Extraspecial $p$-groups}

Next, we consider the family $\mathcal{E}_p$ of finite extraspecial $p$-groups for a fixed prime (along with their subgroups). Recall that an \textit{extraspecial $p$-group} is a $p$-group $G$ such that its center $Z(G)$ is cyclic of order $p$, and the quotient $G/Z(G)$ is a non-trivial elementary abelian $p$-group. Note that it follows that the derived subgroup $G'$ is equal to the center for any extraspecial $p$-group.

For a fixed $p$, there are exactly two extraspecial $p$-groups of order $p^{1+2n}$ which we denote $p^{1+2n}_+$ and $p^{1+2n}_-$. When $p=2$ and $n=1$ we recover familiar groups: $2^{1+2}_+ \cong D_4$ and $2^{1+2}_- \cong Q_8$. We will also use the standard fact that any extraspecial $p$-group can be written as an iterated central product of $p^{3}_-$ and $p^{3}_+$, see \cite[Theorem~5]{Newman}. Recall that for a group $G$, the \emph{exponent} of $G$ is the smallest $n>0$ such that $g^n=1$ for all $g \in G$. We will denote the exponent by $\exp(G)$. For $p$ odd, the group $p^{1+2n}_+$ has exponent $p$, while the exponent of $p^{1+2n}_-$ is $p^2$. In the case $p=2$, we have that all extraspecial 2-groups have exponent 4. We start by determining the images of non-injective endomorphisms of extraspecial $p$-groups.

\begin{lemm}\label{lem:imageofendo}
    Let $G$ be an extraspecial $p$-group. Then the following are equivalent for $\phi \in \mathrm{End}(G)$:
    \begin{enumerate}
        \item $\phi$ is non-injective,
        \item $Z(G) \leq \ker(\phi)$,
        \item $\mathrm{Im}(\phi)$ is an elementary abelian $p$-group.
    \end{enumerate}
\end{lemm}

\begin{proof}
    We begin by proving $(1)\Rightarrow(2)$. It suffices to prove that every non-trivial normal subgroup $N$ of $G$ contains $Z(G)$.

    If $N \leq Z(G)$, then since $|Z(G)|=p$, it follows that $N = Z(G)$. Else, suppose $N \not\leq Z(G)$, that is, there exists some $n \in N$ and $g \in G$ such that $[n,g] \neq 1$. Since $N$ is normal in $G$, we have that $[n,g] \in N$. Moreover $[n,g] \in G' = Z(G)$, and since $|Z(G)| = p$ we get $Z(G) = \langle [n,g] \rangle \leq N$.

    Finally, $(2) \Rightarrow (3)$ follows from the fact that $G/Z(G)$ is elementary abelian by definition, and $(3) \Rightarrow (1)$ is immediate.\qedhere
\end{proof}

Recall from \zcref{defi:endocore} that the \emph{endocore} of a group $G$ is the subgroup generated by the images of non-injective endomorphisms of $G$. In light of \zcref{lem:imageofendo}, and since any elementary abelian $p$-group surjects onto $C_p$, we know that for an extraspecial $p$-group its endocore is precisely the subgroup generated by all the elements of order $p$.  \zcref{lem:endocore} then shows that any a subgroup of an extraspecial group $G$ that contains all elements of order $p$ is endopure. We will now show the converse in a more general setting.

\begin{prop}
\label{prop:order2}
    Let $G$ be a $p$-group of order $> p$ and let $H \leq G$ be endopure in $G$. Then $H$ contains all elements of $G$ of order $p$.
\end{prop}
\begin{proof}
    Assume for a contradiction that we have $x \in G$ of order $p$ and $x \notin H$. Since $H$ is proper in $G$ there exists a maximal subgroup $H \leq M \lneq G$. As $G$ is a finite $p$-group, $M$ is necessarily normal of index $p$. As such the quotient map provides a surjection $q \colon G \to C_p$ with $\ker(q) = M$. As $x$ has order $p$ we can identify $\la x \ra \cong C_p$, and we obtain a non-injective endomorphism $\phi$ of $G$ defined as $\phi(g) = x^{q(g)}$.  Let $k \in H$, then $q(H) = 1$, and hence $\phi(k) = 1 \in G$, giving $H \leq \phi^{-1}(H)$. We get via construction that $\phi^{-1}(H) = \mathrm{ker}(q) = M \lneq G$ and thus $H$ is not endopure.
\end{proof}

\begin{coro}\label{prop:exponents}
    Let $G$ be a finite group with exponent $p$ for $p$ prime. Then 
    \[
    \mathsf{EP}(G) = \begin{cases} \{1\} & G \cong C_p \\\varnothing &\text{else}\end{cases}
    \]
\end{coro}
\begin{proof}
    We know that $1 \leq C_p$ is endopure. \zcref{prop:order2} shows the rest.
\end{proof}

\begin{coro}\label{thm:endoiffcore}
    Let $G$ be an extraspecial $p$-group. Then $H \lneq G$ is endopure if and only if $\Omega(G) \leq H$.
\end{coro}

\begin{proof}
    Let $H \lneq G$ be such that $\Omega(G) \leq H$. By \zcref{lem:endocore} $H$ is endopure. By \zcref{prop:order2}, we see that an endopure subgroup $H$ must contain all elements of order $p$. In particular, we see that $\Omega(G) \leq H$.
\end{proof}

We will now explicitly characterize the endocore of extraspecial $p$-groups for odd $pp$, which will allow us to determine their endopure subgroups by the above result.

\begin{lemm}\label{lem:poddes}
    Let $G$ be an extraspecial $p$-group for $p$ odd. Then
    \[
    \Omega(G) = \{x \in G \mid x^p = 1\} \leq G.
    \]
\end{lemm}

\begin{proof}
    Clearly $\Omega(G)$ is closed under inverses. We need only show it is closed under the operation. To this end let $x,y \in \Omega(G)$. As $G' = Z(G)$ for extraspecial $p$-groups, we see that
    \[
    (xy)^p = x^py^p [y,x]^{\binom{p}{2}}=[y,x]^{\binom{p}{2}}.
    \]
    Again, as $G' = Z(G) \cong C_p$ we have that $[y,x]^p = 1$. As $p$ is odd, $\binom{p}{2}$ is divisible by $p$, thus $[y,x]^{\binom{p}{2}} =1$ and $(xy)^p  = 1$ as required.
\end{proof}

We can go further and explicitly describe the structure of $\Omega(G)$ in the case of exponent $p^2$ (the case of exponent $p$ is trivial):

\begin{lemm}\label{lem:omegaexplicit}
    Let $p$ be an odd prime and let $G \cong p_-^{1+2n}$ be the extraspecial $p$-group of order $p^{1+2n}$ and exponent $p^2$. Then
    \[
    \Omega(G) \cong C_p \times p_+^{1+2(n-1)}. 
    \]
\end{lemm}

\begin{proof}
    We will use the explicit description of $p_-^{1+2n}$ as a central product given in \cite[Theorem~5]{Newman} to determine the elements of order $p$. We have that $p_-^{1+2n} \cong N \circ p_{+}^{1+2(n-1)}$ where $N \cong p_-^{1+2}$ has explicit presentation
    \[
    N \coloneq \la a,b,c \mid a^p=b^p=c,\,c^p=1,\,ba=abc,\,ca=ac,\,cb=bc \ra.
    \]
    From direct calculation, setting $d=ab^{-1}$, we see that $\Omega(N) = \la d,c \ra \cong C_p \times C_p$.

    Moving to $G$, any element of $G$ is of the form $g=xy$ for $x \in N$ and $y \in p_+^{1+2(n-1)}$. As $p_+^{1+2(n-1)}$ is of exponent $p$, we have $y^p=1$, and since $N$ and $p_+^{1+2(n-1)}$ commute (due to $G$ being a central product) we know that $g^p=(xy)^p=x^p$. As such, $g^p=1$ if and only if $x^p=1$ if and only if $x \in \la d,c\ra$. As $c \in p_+^{1+2(n-1)}$ due to the centers being identified, it follows that $\Omega(G) = \la d, p_+^{1+2(n-1)}\ra$. As $d \not\in N \cap p_+^{1+2(n-1)} = \la c \ra$, $\Omega(G)$ splits as a direct product as desired.
\end{proof}

\begin{coro}
    \label{coro:endopure-extra-odd}
    Let $G$ be an extraspecial $p$-group for odd $p$. Then
    \[
    \mathsf{EP}(G) = \begin{cases} \varnothing & G \cong p_+^{1+2n} \\ \{\Omega(G)\} &  G \cong p_-^{1+2n} \end{cases}
    \]
\end{coro}
\begin{proof}
    Apply \zcref{thm:endoiffcore} to \zcref{lem:omegaexplicit} for the second case, and \zcref{lem:poddes} for the first case, since $p_+^{1+2n}$ has exponent $p$.
\end{proof}

Note that the proof of \zcref{lem:poddes} does not work for $p=2$ as $2$ does not divide $\binom{2}{2}=1$, and indeed we know that in $D_4$ we can take the product of involutions and get an element of order 4. As such, we require a different strategy for $p=2$. With the following lemma we will be able to completely determine the endopure subgroups of subgroups of extraspecial 2-groups.

\begin{lemm}\label{lem:ex2noendo}
    Let $H \leq G$ be a non-abelian subgroup of an extraspecial $2$-group $G$. Then $H$ has no endopure subgroups unless $H \cong Q_8$, where $C_2$ and $C_4$ are endopure.
\end{lemm}
\begin{proof}
    For the case that $H \cong Q_8$ we can compute that $\Omega(Q_8) \cong C_2$, and as such it contains endopure subgroups $C_2$, and $C_4$.

    Let $Z \leq G$ denote the center of $G$. Since $H$ is not abelian there exist $a, b \in H$ with $[a, b] \neq e$, so $[a, b] \in Z$, and $Z \leq H$. Since $Z \leq H$ we have that $\Phi(H)=Z=H'$, and $H'=Z \leq Z(H)$, so $H$ is generalized extraspecial in the sense of \cite[Definition~3.1]{MR1887988}. By \cite[Lemma~3.2]{MR1887988} this means that $H=E \times A$ where $E$ is an extraspecial $2$-group and $A$ is an elementary abelian $2$-group, or $H=(E \circ C_4)\times A$. We will show that all of these $H$ are generated by their elements of order $2$, except for when $H \cong Q_8 \times C_2^n$ for some $n \geq 0$. Given \zcref{prop:order2}, this will prove the statement outside of this exceptional case.

    Note that $A=C_2^n$ and $D_4$ are each generated by their order $2$ elements. Recall also that any extraspecial $2$-group can be written as an iterated central product of copies of $D_4$ and $Q_8$ over $C_2$. Given a central product $A \circ B$ of $2$-groups over $C_2$, such that both $A$ and $B$ are generated by their elements of order $2$, it is straightforward to see that the elements of order $2$ of $A \circ B$ generate the images of $A$ and $B$, and thus all of $A \circ B$.

    In $D_4 \circ Q_8$ an element of the form $(1, i)$ can be written as $(\rho, i)(\rho^3\tau, 1)(\tau, 1)$, and $(\rho, i)^2=(\rho^2, -1)=(1, 1)$ in $D_4 \circ Q_8$. Thus $(1, i)$ is a product of elements of order $2$, and $D_4 \circ Q_8$ is generated by its elements of order $2$. The exact same argument applies to $D_4 \circ C_4$. Additionally, note that $Q_8 \circ Q_8 \cong D_4 \circ D_4$ and $Q_8 \circ C_4 \cong D_4 \circ C_4$. Given the associativity of central products over $C_2$, this is enough to show that any group of the form $H=(E \circ C_4)\times A$ is generated by its elements of order $2$, and the same holds for any group of the form $H=E \times A$, unless $E=Q_8$.

    Thus, all that remains is to consider the exceptional case $H \cong Q_8 \times C_2^n$ with $n \geq 1$.  If $K \lneq H$ is endopure it contains all involutions by \zcref{prop:order2}. Thus $K \cong L \times C_2^n$ where $L$ is a proper subgroup of $Q_8$ containing $Z(Q_8) \cong C_2$. Then we can use the non-injective projection $\phi(q,a) = (q,1)$ to reach a contradiction.
\end{proof}
    
\begin{thm}\label{thm:esfinal}\leavevmode
    \[
        |\mathsf{Tr}^\mathsf{Gl}(\mathcal{E}_p)| = \begin{cases}
            5 & p=2 \\
            \infty & p \text{ odd}.
        \end{cases}
    \]
\end{thm}

\begin{proof}
    First, let $p$ be an odd prime. Then by \zcref{coro:endopure-extra-odd} for each $n$ the extraspecial $p$-group $p_-^{1+2n}$ possesses an endopure subgroup, giving an endopure transfer. As such, there are infinitely many meet- and join-irreducible elements of the lattice $\mathsf{Tr}^\mathsf{Gl}(\mathcal{E}_p)$ as required.

    Now, suppose $p=2$. For $n = 0,1$, one obtains the following endopure transfers using \zcref{lem:ex2noendo} and \zcref{prop:abendo}:
    \[ \{1 \leq C_2,\, C_2 \leq C_4,\, C_2 \leq Q_8, \, C_4 \leq Q_8\}\]
    The corresponding formal context matrix for this is
    \[
    \begin{bmatrix}
        0&1&1&1\\
        0&0&0&0\\
        0&1&0&0\\
        0&1&1&0
    \end{bmatrix}
    \]
    from which it follows that $\mathsf{Tr}^\mathsf{Gl}(\mathcal{E}_2) \cong [4]$, giving rise to the 5 claimed global transfers.
\end{proof}

\begin{rem}
    We conclude this section with some discussion about $\mathsf{Tr}^{\mathsf{Gl}}(\mathcal{E}_p(n))$ for $p$ an odd prime. This is the family of extraspecial $p$-groups of order $\leq p^{1+2n}$ along with their subgroups. By \zcref{thm:esfinal} we know that there are infinitely many transfer systems as $n \to \infty$, but it would be nice to know exactly what all the meet- and join-irreducibles are, along with the inclusions between them. We already know that we have the endopure transfers $\Omega(p^{1+2n}_-) \cong (C_p \times p^{1+2(n-1)}_+) \lneq p_-^{1+2n}$, but the question remains if there are any other transfers. In particular, we are interested in subgroups $H$ of $p_-^{1+2n}$ such that $H$ possesses endopure transfers. By \zcref{prop:exponents} it is enough to consider those $H$ of exponent $p^2$, which is a very restrictive condition. 

    Using \texttt{GAP} we are able to identify for each $n$ another endopure transfer
    \[
    \Omega( p_+^{1+2(n-1)} \circ C_{p^2}) \cong p_+^{1+2(n-1)} \lneq ( p_+^{1+2(n-1)} \circ C_{p^2})
    \]
    and we conjecture that this is the \emph{only} new endopure transfer for each $n$. Following through the inclusions of the generated and antigenerated transfer systems, if our conjecture is true, then one can compute that
    \[
        \mathsf{Tr}^{\mathsf{Gl}}(\mathcal{E}_p(n)) \cong [2n+1].
    \]
    In particular, we see that the corresponding lattice of global transfer systems is again linear, and as $n \to \infty$ we have $\mathsf{Tr}^{\mathsf{Gl}}(\mathcal{E}_p) \cong \omega+1$.
\end{rem}

\subsection{Finite groups of bounded order}\label{subsec:allgroups}

We can extend our families of finite abelian groups to families of \textit{all} finite groups of order less than or equal to $n$, which we denote by $\mathcal{G}_n$. In \zcref{tab:ep} we report the first handful of values for this family; the number of endopure transfers, the density of the formal context matrix, and the number of transfers. Of course, given the complexity of finite group theory, we do not hope to be able to give any formula for this situation, but the fact that we can do this highlights the strength of the techniques developed in the paper, and doing so leads to some intriguing observations.

\begin{table}[h]
\begin{tabular}{c|c|c|c}
$n$ & $|\mathsf{EP}(\mathcal{G}_n)|$ & $\delta(\mathcal{G}_n)$ & $|\mathsf{Tr}^\mathsf{Gl}(\mathcal{G}_n)|$ \\ \hline
2 & 1 & 0 & 2 \\
3 & 2 & 0.5 & 4 \\
4 & 3 & 0.5556 & 6 \\
5 & 4 & 0.6875 & 12 \\
6 & 5 & 0.64 & 14 \\
7 & 6 & 0.7222 & 28 \\
8 & 9 & 0.6914 & 52 \\
9 & 10 & 0.72 & 76 \\
10 & 11 & 0.7107 & 84 \\
11 & 12 & 0.75 & 168 \\
12 & 14 & 0.7449 & 240 \\
13 & 15 & 0.7733 & 480 \\
14 & 16 & 0.7734 & 576 \\
15 & 16 & 0.7734 & 576 \\
16 & 25 & 0.72 & 2016 \\
17 & 26 & 0.7396 & 4032 \\
18 & 27 & 0.7353 & 4392 \\
19 & 28 & 0.7526 & 8784 \\
20 & 30 & 0.7567 & 13104
\end{tabular}
\caption{A table listing the number of endopure transfers, the density, and number of global transfers for the family of all finite groups of order $\leq n$.}\label{tab:ep}
\end{table}

\begin{ex}
    As we have seen many cases where linear orders (and products of such) appear as transfer lattices, let us consider the case of $\mathcal{G}_6$. We know from the previous section that $\mathsf{Tr}^\mathsf{Gl}(\mathcal{A}_6) \cong [2] \times [1] \times [1]$. The non abelian case is rather different:   $\mathsf{Tr}^\mathsf{Gl}(\mathcal{G}_6) \cong [1] \times L$ where
    \[L \cong 
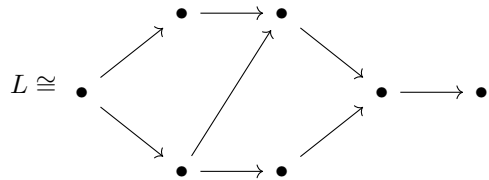
\begin{tikzcd}
	& \bullet & \bullet && \\
	\bullet &&& \bullet & \bullet \\
	& \bullet & \bullet
	\arrow[from=1-2, to=1-3]
	\arrow[from=1-3, to=2-4]
	\arrow[from=2-1, to=1-2]
	\arrow[from=2-1, to=3-2]
	\arrow[from=2-4, to=2-5]
	\arrow[from=3-2, to=1-3]
	\arrow[from=3-2, to=3-3]
	\arrow[from=3-3, to=2-4]
\end{tikzcd}\]
\end{ex}

\begin{rem}
    We have seen above that in the abelian case, the density as $n \to \infty$ approaches $1$. One could perhaps expect the same for $\delta(\mathcal{G}_n)$, but this does not appear to be the case. \zcref{fig:gnuplot} provides a plot of the density values for $3 \leq n \leq 64$. The general fact that $(L \leq K) \leq_{\mathsf{EP}} (H \leq G)$ whenever there are no non-trivial homomorphisms $G \to K$ produces the general upwards trend. But when $n=2^k$ we see a steep drop in the density. This is due to the sheer amount of groups of order $2^k$  giving rise to many new endopure transfers (c.f., \zcref{fig:gnep}) which are not necessarily related under $\leq_{\mathsf{EP}}$ (compare with \zcref{ex:density-cyclic}). From this (minimal) computational evidence, one may cautiously conjecture that the limiting density is strictly bounded above by $1$ and below by $0.5$, with the value being perhaps $\approx 0.75$. The authors see no reason why this density should limit to some known constant. Note that this does not contradict the density calculation of \zcref{subsec:ab} as almost all groups are non-abelian.
    \begin{figure}[h]
    \centering
    \begin{tikzpicture}
\begin{axis}[
    width=12cm,
    height=8cm,
    xlabel={$n$},
    ylabel={$\delta(\mathcal{G}_n)$},
    xmin=1.5,
    xmax=64.5,
    ymin=0.5,
    ymax=0.85,
    grid=major,
    scatter,
    only marks,
    mark=*,
    mark size=1.5pt,
    scatter/use mapped color={red!80!black}
]
\addplot[
    scatter,
] coordinates {
    (3,0.5)
    (4,0.55555556)
    (5,0.6875)
    (6,0.64)
    (7,0.72222222)
    (8,0.69135802)
    (9,0.72)
    (10,0.7107438)
    (11,0.75)
    (12,0.74489796)
    (13,0.77333333)
    (14,0.7734375)
    (15,0.7734375)
    (16,0.72)
    (17,0.73964497)
    (18,0.73525377)
    (19,0.75255102)
    (20,0.75666667)
    (21,0.76795005)
    (22,0.76660156)
    (23,0.77961433)
    (24,0.78172255)
    (25,0.789375)
    (26,0.79000595)
    (27,0.79394267)
    (28,0.79958678)
    (29,0.80790123)
    (30,0.80790123)
    (31,0.81569)
    (32,0.747701)
    (33,0.74770146)
    (34,0.74606949)
    (35,0.74606949)
    (36,0.74481481)
    (37,0.75027171)
    (38,0.74905482)
    (39,0.75372875)
    (40,0.76226128)
    (41,0.76703157)
    (42,0.765202)
    (43,0.76981941)
    (44,0.7729)
    (45,0.7729)
    (46,0.77188511)
    (47,0.77623991)
    (48,0.77532124)
    (49,0.77847356)
    (50,0.77807617)
    (51,0.77807617)
    (52,0.78023669)
    (53,0.78352077)
    (54,0.78218102)
    (55,0.78519715)
    (56,0.78890724)
    (57,0.79164041)
    (58,0.79147042)
    (59,0.79438776)
    (60,0.78424032)
    (61,0.78708166)
    (62,0.7872)
    (63,0.789702)
    (64,0.743977)
};
\end{axis}
\end{tikzpicture}
\caption{A scatter plot showing the density of the formal context matrix corresponding to the global transfer systems for $\mathcal{G}_n$.}\label{fig:gnuplot}
\end{figure}
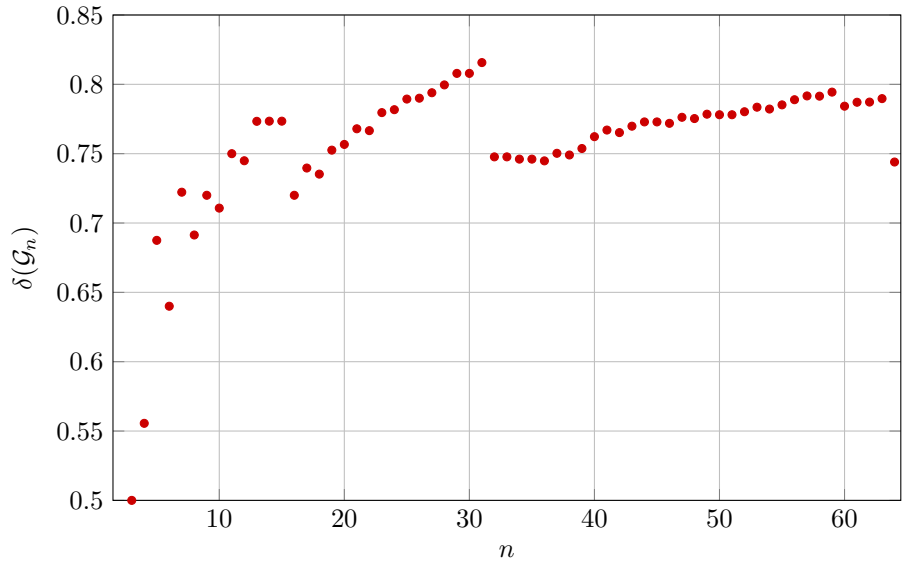

    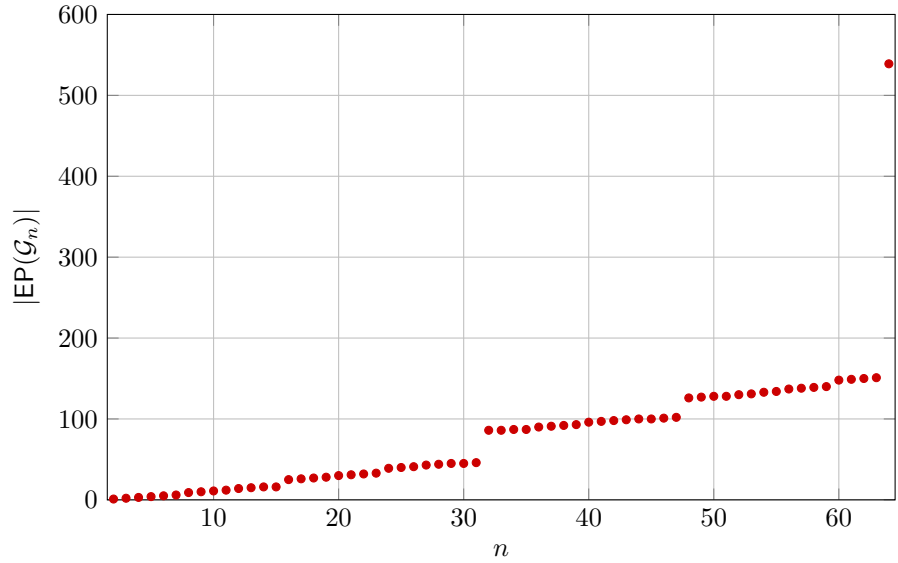
\begin{figure}[h]
    \centering
    \begin{tikzpicture}
\begin{axis}[
    width=12cm,
    height=8cm,
    xlabel={$n$},
    ylabel={$|\mathsf{EP}(\mathcal{G}_n)|$},
    xmin=1.5,
    xmax=64.5,
    ymin=0,
    ymax=600,
    grid=major,
    scatter,
    only marks,
    mark=*,
    mark size=1.5pt,
    scatter/use mapped color={red!80!black}
]
\addplot[
    scatter,
] coordinates {
(2,1)
    (3,2)
    (4,3)
    (5,4)
    (6,5)
    (7,6)
    (8,9)
    (9,10)
    (10,11)
    (11,12)
    (12,14)
    (13,15)
    (14,16)
    (15,16)
    (16,25)
    (17,26)
    (18,27)
    (19,28)
    (20,30)
    (21,31)
    (22,32)
    (23,33)
    (24,39)
    (25,40)
    (26,41)
    (27,43)
    (28,44)
    (29,45)
    (30,45)
    (31,46)
    (32,86)
    (33,86)
    (34,87)
    (35,87)
    (36,90)
    (37,91)
    (38,92)
    (39,93)
    (40,96)
    (41,97)
    (42,98)
    (43,99)
    (44,100)
    (45,100)
    (46,101)
    (47,102)
    (48,126)
    (49,127)
    (50,128)
    (51,128)
    (52,130)
    (53,131)
    (54,133)
    (55,134)
    (56,137)
    (57,138)
    (58,139)
    (59,140)
    (60,148)
    (61,149)
    (62,150)
    (63,151)
    (64,539)
};
\end{axis}
\end{tikzpicture}
\caption{A scatter plot showing the number of endopure transfers for the family $\mathcal{G}_n$.}\label{fig:gnep}
\end{figure}
\end{rem}

\begin{rem}[Ghost orders]
    We see a curio in \zcref{tab:ep} in that the rows for 14 and 15 are repeated. This is due to the fact that the only group of order 15 is the cyclic group $C_{15}$ which introduces no new endopure transfers.  We will say that $n \in \mathbb{N}$ is a \emph{ghost order} if there are no groups of order $n$ with an endopure subgroup.
    
    We say that $n = p_1^{e_1} \cdots p_r^{e_r} \in \mathbb{N}$ is an \emph{abelian number} if $e_i \leq 2$ for all $i$ and $p_i^k$ is not congruent to $1 \mod{p_j}$ for all $i$ and $j$ and $1 \leq k \leq e_i$. Then it is well known that any group with order $n$ is necessarily abelian.

    Clearly if $n$ is a non-prime power abelian number then $n$ is a ghost order. What may be somewhat of a surprise, is that these non-prime power abelian numbers are not the only ghost orders. 
    
    Say that a group is \textit{indecomposable} if it is not of the form $G_1 \times G_2$ for $G_i$ non-trivial. If there are no groups of order $n$ which are indecomposable, then we say that $n$ is decomposable. Then  work of Erd\H os--P\'{a}lfy in \cite{MR1692288} studies the sequence of decomposable numbers. Note that this sequence includes all non-prime-power abelian numbers, so subsumes the above sequence.
    
    We have computationally verified that every entry of \cite{oeisA243592} is a ghost order. This is extremely surprising in light of \zcref{ex:c4xq8} which tells us that direct product groups \textit{can} possess endopure subgroups. However, given the computational evidence one would conjecture that every decomposable number is a ghost order. Pairing this with the observation from \cite{MR1692288} that essentially all odd numbers are decomposable, we could conclude that essentially all odd numbers will be ghost orders.
    
    Even more surprising is that there are ghost orders which are not captured by any of the discussion above. For example $n=30$ is a ghost order. The groups of order 30 are $C_{30}, D_{15},C_5 \times S_3, C_3 \times D_5$ and none of these have endopure subgroups.

    For posterity, we list all ghost orders $<500$ below, with those in red being those which do not come from decomposable numbers. It appears to be the case, at least with this evidence, that all ghost orders are of the form $2^kn$ where $n$ is a decomposable number, but not all numbers of this form are ghost orders (for example $60$ is not a ghost order, due to the group $A_5$).

15, \textcolor{red!80!black}{30}, 33, 35, 45, 51, 65, \textcolor{red!80!black}{66}, 69, \textcolor{red!80!black}{70}, 77, 85, 87, \textcolor{red!80!black}{90}, 91, 95, 99, \textcolor{red!80!black}{102}, 105, 115, 119, 123, \textcolor{red!80!black}{130}, \textcolor{red!80!black}{132}, 133, 135, \textcolor{red!80!black}{138}, \textcolor{red!80!black}{140}, 141, 143, 145, 153, \textcolor{red!80!black}{154}, 159, 161, 165, \textcolor{red!80!black}{170}, \textcolor{red!80!black}{174}, 175, 177, \textcolor{red!80!black}{182}, 185, 187, \textcolor{red!80!black}{190}, 195, \textcolor{red!80!black}{198}, \textcolor{red!80!black}{204}, 207, 209, \textcolor{red!80!black}{210}, 213, 215, 217, 221, \textcolor{red!80!black}{230}, 231, 235, \textcolor{red!80!black}{238}, 245, \textcolor{red!80!black}{246}, 247, 249, 255, 259, \textcolor{red!80!black}{260}, 261, \textcolor{red!80!black}{264}, 265, \textcolor{red!80!black}{266}, 267, \textcolor{red!80!black}{270}, \textcolor{red!80!black}{273}, \textcolor{red!80!black}{276}, \textcolor{red!80!black}{280}, \textcolor{red!80!black}{282}, 285, \textcolor{red!80!black}{286}, 287, \textcolor{red!80!black}{290}, 295, 297, 299, 303, \textcolor{red!80!black}{306}, \textcolor{red!80!black}{308}, 315, \textcolor{red!80!black}{318}, 319, 321, \textcolor{red!80!black}{322}, 323, 325, 329, \textcolor{red!80!black}{330}, 335, 339, \textcolor{red!80!black}{340}, 341, 345, \textcolor{red!80!black}{348}, \textcolor{red!80!black}{350}, \textcolor{red!80!black}{354}, 357, \textcolor{red!80!black}{364}, 365, 369, \textcolor{red!80!black}{370}, 371, \textcolor{red!80!black}{374}, 377, \textcolor{red!80!black}{380}, 385, \textcolor{red!80!black}{390}, 391, 393, 395, \textcolor{red!80!black}{396}, \textcolor{red!80!black}{399}, 403, 407, \textcolor{red!80!black}{408}, 411, 413, \textcolor{red!80!black}{414}, 415, \textcolor{red!80!black}{418}, \textcolor{red!80!black}{420}, 423, 425, \textcolor{red!80!black}{426}, 427, 429, \textcolor{red!80!black}{430}, \textcolor{red!80!black}{434}, 435, 437, \textcolor{red!80!black}{442}, 445, 447, 451, 455, 459, \textcolor{red!80!black}{460}, \textcolor{red!80!black}{462}, 469, \textcolor{red!80!black}{470}, 473, 475, \textcolor{red!80!black}{476}, 477, 481, 483, 485, \textcolor{red!80!black}{490}, \textcolor{red!80!black}{492}, 493, \textcolor{red!80!black}{494}, 495, \textcolor{red!80!black}{498}.
\end{rem}

\newpage

\section*{Appendix. \texttt{GAP} code}\hypertarget{appendix}{}

\begin{lstlisting}[language=GAP]
# A function which checks if a subgroup H of G is endopure
# Returns true if H is endopure in G, and false else
isEndopure := function(G, H)
    local normals, subgroups, N, K, Q, HQ, I, f0, autK, alpha, nat, HQimage, KK;
    
    if(Size(H) = Size(G)) then
        return false;
    fi;
    
    normals := NormalSubgroups(G);
    subgroups := ConjugacyClassesSubgroups(G);

    # Loop over proper, non-trivial normal subgroups N
    for N in normals do
        if N <> TrivialSubgroup(G) and N <> G then

            # Natural quotient map G -> G/N
            nat := NaturalHomomorphismByNormalSubgroup(G, N);

            # Use the image of this map as the quotient group
            Q := Image(nat);

            # Image of H in G/N (i.e., HN/N)
            HQ := Image(nat, H);

            # Loop over proper subgroups K not contained in H
            for KK in subgroups do
                K := Representative(KK);
                if K <> G and not IsSubgroup(H, K) then

                    # K intersect H
                    I := Intersection(K, H);

                    # Check whether G/N is isomorphic to K
                    f0 := IsomorphismGroups(Q, K);

                    if f0 <> fail then

                        # Image of HN/N under f0
                        HQimage := Image(f0, HQ);
                        
                        # Look for an aut of K taking HQimage onto K intersect H
                        autK := AutomorphismGroup(K);

                        for alpha in Elements(autK) do
                            if IsSubset(I,Image(alpha, HQimage)) then
                                # A matching pair has been found
                                return false;
                            fi;
                        od;
                    fi;
                fi;
            od;
        fi;
    od;
    return true;
end;
\end{lstlisting}

\newpage

\begin{lstlisting}[language=GAP]
# A function which checks if subgroups H and K are isomorphic via an automorphism
# Returns true if H and K are isomorphic, and false else
AreEquivalent := function(H,K,A)
    return ForAny(Elements(A), phi -> Image(phi,H) = K);
end;


# A function which returns EP(G)
# Returns pairs [H,G] where H is endopure in G
endopureSubgroups := function(G)
    local classes, res, pair, H, C, A;
    
    res := [];
    
    classes := ConjugacyClassesSubgroups(G);
    A := AutomorphismGroup(G);
    
    for C in classes do
        H := Representative(C);
        
        # Check if H is endopure in G
        if(isEndopure(G,H)) then
            pair := [H,G];
            # Check if an automorphic copy of H already has been found
            if ForAll(res, rep -> not AreEquivalent(H, rep[1], A)) then
                Add(res, pair);
            fi;
        fi;
    od;
    return res;
end;
\end{lstlisting}

\begin{lstlisting}[language=GAP]
# Computing all endopure subgroups of Q(64)
G:=SmallGroup(64,54);
for S in endopureSubgroups(G) do
    Print(StructureDescription(S[1]),"<",StructureDescription(S[2]), "\n");
od;

# Output
C2<Q64
C4<Q64
C4<Q64
C8<Q64
Q8<Q64
C16<Q64
Q16<Q64
Q32<Q64
C32<Q64
\end{lstlisting}

\newpage
\printbibliography

\end{document}